\documentclass[11pt]{amsart}
\usepackage{amscd,amsmath,amssymb,amsfonts,verbatim}
\usepackage[cmtip, all]{xy}
\usepackage{MnSymbol}
\usepackage{comment}

\newtheorem{thm}{Theorem}[section]
\newtheorem{prop}[thm]{Proposition}
\newtheorem{lem}[thm]{Lemma}
\newtheorem{cor}[thm]{Corollary}

\theoremstyle{definition}

\newtheorem{defn}[thm]{Definition}
\theoremstyle{remark}
\newtheorem{remk}[thm]{Remark}
\newtheorem{remks}[thm]{Remarks}

\newtheorem{exm}[thm]{Example}
\newtheorem{exms}[thm]{Examples}
\newtheorem{notat}[thm]{\bf{Notation}}

\numberwithin{equation}{section}

{\hfill$\square$\end{defn}}
{\hfill$\square$\end{remk}}
{\hfill$\square$\end{remks}}
{\hfill$\square$\end{exm}}
{\hfill$\square$\end{exms}}
{\hfill$\square$\end{notat}}

\newcommand{\thmref}{Theorem~\ref}
\newcommand{\propref}{Proposition~\ref}
\newcommand{\corref}{Corollary~\ref}

\newcommand{\lemref}{Lemma~\ref}

\newcommand{\remref}{Remark~\ref}

\newcommand{\notbref}{Notation~\ref}

\newcommand{\sI}{{\mathcal I}}

\newcommand{\sO}{{\mathcal O}}
\newcommand{\sP}{{\mathcal P}}
\newcommand{\sQ}{{\mathcal Q}}

\newcommand{\sU}{{\mathcal U}}
\newcommand{\sV}{{\mathcal V}}

\newcommand{\sX}{{\mathcal X}}
\newcommand{\sY}{{\mathcal Y}}

\newcommand{\F}{{\mathbb F}}

\renewcommand{\P}{{\mathbb P}}
\newcommand{\Q}{{\mathbb Q}}

\newcommand{\Z}{{\mathbb Z}}

\newcommand{\fm}{{\mathfrak m}}

\newcommand{\surj}{\twoheadrightarrow}
\newcommand{\inj}{\hookrightarrow}
\newcommand{\red}{{\rm red}}

\newcommand{\Spec}{{\rm Spec \,}}
\newcommand{\sing}{{\rm sing}}

\newcommand{\Sch}{{\operatorname{\mathbf{Sch}}}}

\renewcommand{\max}{{\operatorname{\rm max}}}

\newcommand{\ds}{{/\kern-3pt/}}

\newcommand{\sm}{{\operatorname{sm}}}
\newcommand{\Proj}{{\operatorname{Proj}}}

\newcommand{\un}{\underline}
\newcommand{\ov}{\overline}

\renewcommand{\dim}{\text{\rm dim}}

\newcommand{\tuborg}{\left\{\begin{array}{ll}}
\newcommand{\sluttuborg}{\end{array}\right.}

\newcommand{\edim}{{\rm edim}}
\newcommand{\tdim}{{\rm tdim}}
\newcommand{\homg}{{\rm homog}}

\newcommand{\irr}{{\rm Irr}}
\newcommand{\reg}{{\rm reg}}

\newcommand{\wt}{\widetilde}

\newcommand{\nor}{{\rm nor}}

\newcounter{elno}

\newcounter{elno-abc}   
\newenvironment{listabc}{
                         \begin{list}{\alph{elno-abc})
                                     }{\usecounter{elno-abc}}
                      }{
                         \end{list}}
\newcounter{elno-abc-prime}

\begin{document}
\title{Bertini theorems revisited}
\author{Mainak Ghosh, Amalendu Krishna}
\address{Department of Mathematical Sciences, Indian Institute of Science
Education and Research, Mohali, 140306, India.}
\email{mainak.09.13@gmail.com}
\address{Department of Mathematics, Indian Institute of Science,  
Bangalore, 560012, India.}
\email{amalenduk@iisc.ac.in}

%\dedicatory{}

\keywords{Bertini theorems, Schemes over arbitrary fields and dvr}        

\subjclass[2010]{Primary 14J17; Secondary 14C25, 14G15}

\maketitle

\begin{quote}\emph{Abstract.}  
We prove several new Bertini theorems over arbitrary fields and
discrete valuation rings.
\end{quote}
%\end{abstract}
\setcounter{tocdepth}{1}
\maketitle
\tableofcontents  

%\maketitle
%\tableofcontent

\section{Introduction}\label{sec:Intro}
\subsection{Background}
Let $\sP$ be a property of schemes
(e.g., regular, smooth, normal, weakly normal, semi-normal,
strongly F-regular, F-pure, integral, reduced and so on).
In the present language of algebraic geometry, a Bertini
theorem for $\sP$ (sometimes called {\sl the Bertini-$\sP$ theorem})
broadly says that if a subscheme of a projective space over a base scheme
satisfies the property $\sP$, then `almost all' hypersurface
sections of the subscheme inside the projective space also satisfy $\sP$.
The Bertini theorems are known to be very powerful tools in the study
of algebraic varieties.  
They play a very important role in reducing a problem about higher
dimensional varieties to curves and surfaces.

The Bertini theorem for smoothness over infinite fields is a classical
result. This problem over finite fields was solved by Gabber \cite{Gabber}
in a limited form (in particular, for hypersurfaces of degrees large enough but
divisible by the characteristic of the base field), and
Poonen \cite{Poonen-1} in general form
(a vastly improved probabilistic result for hypersurfaces of all large degrees).
Using his `geometric closed point sieve method', Poonen \cite{Poonen-2} also proved 
Bertini-smoothness theorems over finite fields for hypersurfaces containing a
given closed subscheme satisfying certain conditions.
This result of Poonen was later generalized by Gunther \cite{Gunther}
and Wutz \cite{Wutz}. The analogous results over
infinite fields were earlier obtained by Bloch \cite{Bloch} and 
Altman-Kleiman \cite{KA}.
In another work, Charles-Poonen \cite{Poonen-3} used Poonen's
techniques to prove a Bertini-irreducibility theorem over finite fields.
There have been other generalizations and applications of Poonen's Bertini theorems
in the recent past.

However, apart from smoothness and irreducibility,
Bertini theorems are still mostly unknown for other properties of schemes
one frequently encounters in algebraic geometry, especially over base rings which are
not perfect fields.
The objective of this paper is to prove some of these Bertini 
theorems over fields and discrete valuation rings.

\subsection{Main results}\label{sec:Results}
Our main results are the Bertini theorems for 
regularity, normality, reducedness, irreducibility and integrality
over (possibly imperfect) fields and discrete valuation rings.
Furthermore, we prove these results in the generality in which
the hypersurfaces are required to contain a prescribed 
closed subscheme (satisfying some necessary conditions) of the ambient 
projective space.
As consequences of this flexibility, we obtain generalizations of the 
Bertini irreducibility 
theorem over finite fields by Charles-Poonen and the Bertini-normality theorem over
infinite fields by Seidenberg.

Before we describe the main results, we discuss some concrete motivations
for proving them. Although the Bertini theorems of different kinds are
known to be powerful tools in algebraic geometry and will remain so in the
future, our interest in proving them was mainly driven by some precise (and some potential)
applications in the theory of algebraic cycles. Typically, these applications are
in finding a hypersurface section of a higher dimensional algebraic variety
which contains a given algebraic cycle. This allows one
to reduce the given problem to the case of algebraic cycles on lower dimensional varieties.

If the base field $k$ is not perfect and a quasi-projective $k$-scheme is regular but
not smooth over $k$, then the Bertini smoothness theorem for hypersurfaces
containing a closed subscheme is not helpful in solving the
above problem. Over such a field, we also often need Bertini theorems for
other properties such as reducedness and integrality.
For instance, the Bertini theorems of this paper (for regularity, reducedness and
integrality) over imperfect fields
are crucially used in the proof of \cite[Lemma~3.2, p.~12]{GKR-1}.

The Bertini theorems for normal crossing and singular schemes (for hypersurfaces
containing a closed subscheme and avoiding another closed subscheme) over imperfect fields
are key steps in the proofs of \cite[Proposition~7.7, Theorem~8.3]{BK-1}
(see p.~317 and p.~321 of op. cit.).
For another application of the Bertini theorem for normal crossing schemes,
we refer to the proof of \cite[Lemma~3.3]{BK-TAMS}. It has also been realized that
the Bertini regularity theorem for hypersurfaces containing a 0-cycle is a key
requirement in the potential generalizations of
the main results of \cite{BK-TAMS} to regular schemes over imperfect base fields.
For applications of the Bertini theorems for normality and $(R_a + S_b)$-property,
we refer the reader to the proofs of Theorem~1.2 (p.~33) and Theorem~8.6 (p.~46) of
\cite{Ghosh-Krishna}.

In the study of class field theory for varieties over a local field $k$, one has to usually
work with a model of a given $k$-variety over the ring of integers of $k$.
Such a model is not guaranteed to be smooth even if the given variety over $k$ is.
In such cases, one needs very general Bertini theorems for quasi-projective schemes
over discrete valuation rings. When the model is smooth, a Bertini theorem
of this kind has already played a key role in the proof of \cite[Lemma~13.2]{GKR}. Our hope is
that the general Bertini theorem of this paper will be very useful in the study of
class field theory of regular varieties over local fields of positive characteristics.
%Nothing is known about this at present due to the lack of such Bertini theorems.
For more applications of the results of this paper, the reader can also see
\cite{BKS} and \cite{GK-1}.

\vskip .3cm

We now summarize our main results and refer to various sections for
precise statements.

\subsubsection{Bertini for regularity}
Our first set of results consists of Bertini theorems for regularity for hypersurface sections containing a prescribed closed subscheme over arbitrary infinite fields (see \propref{prop:B-regularity-*}). Note that this is stronger than the Bertini smoothness theorem when the field under consideration is imperfect.

One consequence of the new Bertini-regularity theorems is that they
allow a direct extension of the `strong Bertini theorems' of Diaz-Harbater
\cite{DH} to arbitrary infinite base fields. Interested readers can check
\cite[Theorem~4.1]{GK-arxiv} for a proof of this. We omit this part from this paper
for brevity.

\subsubsection{Bertini for normality}
The Bertini problem for normality was raised 
for the first time in a joint paper of H. Muhly and O. Zariski
\cite{MZ} as part of their attempt to prove resolution of
singularities. 
The question of Muhly-Zariski over infinite fields was answered by Seidenberg
\cite{Seidenberg}. However, this question does not yet have an
answer over finite fields. 
We provide an answer to the question of Muhly-Zariski over finite fields.
In our Bertini-normality theorem, the hypersurface sections are moreover required to contain a prescribed closed subscheme with some condition.
In particular, we obtain a stronger version of Seidenberg's Bertini theorem
over infinite fields. See Theorems~\ref{thm:Bertini-reg-0},
~\ref{thm:B-smooth-inf} and \corref{cor:B-reduced-fin}.

\subsubsection{Bertini for reducedness and integrality}
The Bertini theorems for geometrically reduced and geometrically
integral schemes are known over infinite fields (e.g.,
see \cite[Theorem~1]{KA} and \cite[Th{\'e}or{\'e}me~6.3]{Jou}).
A version of Bertini theorem for the geometric integrality 
for a family of projective schemes
over arbitrary fields is due to Benoist \cite{Ben}.
However, such results are not available today if the given variety is only
reduced (or integral), especially if we ask our hypersurfaces
to contain a prescribed closed subscheme.  
Our results resolve these problems.
See Theorems~\ref{thm:Bertini-reg-0},
~\ref{thm:B-smooth-inf}, ~\ref{thm:B-int-fin} and \corref{cor:B-reduced-fin}.

\subsubsection{Bertini for schemes over a dvr}
At present, we do not have many Bertini type theorems for properties like regularity,
smoothness, reducedness and integrality for hypersurface sections of 
quasi-projective schemes over a discrete valuation ring.
As noted above, such theorems are very useful in the study of
class field theory and algebraic cycles on varieties over higher local fields.

The Bertini-regularity theorem for schemes which are regular, flat and projective over 
a discrete valuation ring with normal crossing special fiber was proven by 
Jannsen-Saito \cite{JS} and Saito-Sato \cite{SS}. The normal crossing special fiber
condition was recently removed by Binda-Krishna \cite{BK} under the condition that the
residue field is infinite and perfect.
A form of Bertini-normality theorem for affine and flat normal schemes
over a discrete valuation ring was obtained by Horiuchi-Shimomoto
\cite{HS} under some conditions on the ring.
To our knowledge, apart from the above results, 
no other Bertini type result seems to be known for 
schemes over a discrete valuation ring. 
In this paper, we establish many of these results
(see Theorems~\ref{thm:B-base-reg} and ~\ref{thm:B-Bertini-red}). In particular,
we generalize the
results of Jannsen-Saito, Saito-Sato and Binda-Krishna to
arbitrary quasi-projective schemes.

It would be interesting to obtain such Bertini theorems for 
quasi-projective schemes over more general bases than the
spectrum of a discrete valuation ring. Some results of this nature and their
applications have been obtained in the past by Gabber-Liu-Lorenzini
(see \cite[\S~2]{GLL-1} and \cite[\S~3.]{GLL}).

\subsection{Organization of the paper}\label{sec:Outline}
We prove some Bertini-regularity theorems over infinite (in particular, imperfect)
fields in \S~\ref{sec:Prelim}. The Bertini theorems for other properties over
infinite fields are shown in \S~\ref{sec:Inf**}.
In \S~\ref{sec:BNT*}, we prove the
Bertini theorems for normality and reducedness over finite fields.
In this section, we generalize the Bertini theorems of Poonen
\cite{Poonen-1}
and Wutz \cite{Wutz} over finite fields and give new applications.
These are used in the proofs of Bertini theorems over a discrete valuation ring.
In \S~\ref{sec:Irr-2}, we prove the Bertini-irreducibility and
integrality theorems over finite fields.
In \S~\ref{sec:dvr*}, we prove Bertini theorems for quasi-projective
schemes over a discrete valuation ring. These are obtained by suitably
combining the Bertini theorems over
the quotient field and the residue field of the discrete valuation ring.

\begin{notat}\label{notat:Notn}
%\subsection{Notations}\label{sec:Notn}
Given a field $k$, we shall let $\Sch_k$ denote the
category of finite type separated schemes over $k$. 
If $k \subset k'$ is a field extension,
we shall let $X_{k'} = X \times \Spec(k')$.
For a Noetherian scheme $X$, we shall let $\irr (X)$ denote the
set of all irreducible components of $X$.
We shall let $X_\red$ denote the closed subscheme of $X$ defined
by the ideal sheaf of nilradicals of $\sO_{X}$.
An intersection of two or more subschemes of $X$ will mean a
scheme theoretic intersection unless we mention otherwise.
For a scheme $X$ over $k$ and a field extension $k \subset k'$, we shall let
$X(k')$ denote the set of morphisms $\Spec(k') \to X$ in the category
of all schemes over $k$.

By a subscheme of $\P^n_k = \Proj(k[x_0, \ldots , x_n])$, we shall mean a $k$-scheme $X$
with a locally closed embedding $X \subset \P^n_k$.  For such a subscheme, we shall let
$\ov{X}$ denote the scheme-theoretic closure of $X$ in $\P^n_k$.
If $X \subset \P^n_k$ is a closed subscheme, the notation $\sI_X$ will mean the
sheaf of ideals on $\P^n_k$ which defines $X$.
If $H_f \subset \P^n_k = \Proj(k[x_0, \ldots , x_n])$ is the hypersurface
defined by a homogeneous polynomial $f$, and if $k \subset k'$ is a field
extension, we shall usually denote $(H_f)_{k'}$ simply by $H_f$ whenever
$k'$ is given.
For $X \subset \P^n_{k'}$ (or a subset of $\P^n_k$)
and $f \in k[x_0, \ldots , x_n]$ homogeneous,
we shall write $X \cap (H_f)_{k'}$ (or $X \cap H_f$) as $X_f$.
\end{notat}

%\notref{notat:Notn}

\section{Bertini for regularity over infinite fields}\label{sec:Prelim}
Recall that for a Noetherian scheme $X$, the subset of points of $X$ 
whose local rings are regular is denoted by $X_\reg$.
If $X$ is excellent (e.g., objects of $\Sch_k$ where $k$ is a field), 
$X_\reg$ has the canonical structure of an open
subscheme of $X$ which is dense if $X$ is generically reduced.
If $X \in \Sch_k$ for a field $k$, we let $X_\sm \subset X$ denote the smooth
locus of $X$. This is the largest open
subscheme of $X$ at each of whose points $X$ is smooth over $k$.
Recall that $X_\sm \subset X_\reg$ and the equality holds if $k$ is perfect.
We let $X_\sing$ be the complement of $X_\reg$ with the reduced closed subscheme
structure. We shall use the following notation and definition throughout this paper.

\begin{notat}\label{notat:Delta-X}
For a Noetherian scheme $X$, we let $\Delta(X)$ be the subset of $X$ such that
$x \in \Delta(X)$ if and only if $x$ is either a generic point of $X$
or an associated point of $X$ or a generic point of $X_\sing$.
Note that the associated and generic points of $X$ coincide if it is reduced
(e.g., see \cite[Lemma~3.3]{GuK}).
\end{notat}
%We remark here that if $X$ is reduced, then the associated and generic points of $X$ coincide (e.g., see \cite[Lemma~3.3]{GuK}).

%\begin{assum}\label{assum:Defn}
%Test
%\end{assum}
%\asref{assum:Defn}

\begin{defn}\label{defn:Gen}
  Let $\sP$ be a property of schemes.
  If $k$ is an infinite field and $X, Z \subset \P^n_k$ are subschemes such that $Z$ is
  closed, then we shall say that $X \cap H$ satisfies $\sP$ for  
a general hypersurface $H \subset \P^n_k$ containing $Z$  
if for all $d \gg 0$, there is a nonempty open 
subscheme $\sU$ of the linear system $|H^0(\P^n_k, \sI_Z(d))|$
such that $X \cap H$ satisfies $\sP$ for all $H \in \sU(k)$.
\end{defn}

\begin{comment}
Let $Z, T \subset \P^n_k$ be two subschemes such that $X \cap Z \cap T = \emptyset$.
Let $H \subset \P^n_k$ be a hypersurface containing $Z$ and
consider the following conditions.
\begin{enumerate}
\item
$T \cap X \cap H = \emptyset$.
\item
$X \cap H$ is an effective Cartier divisor on $X$.
\item
$X \cap H$ does not contain any irreducible component of $X_\sing$.
\item
$X_\reg \cap H$ is regular.
\item
$X_\sm \cap H$ is smooth.
\end{enumerate}

Throughout this paper, we shall refer to the $i$-th property
listed above as (Gi). We shall say that $X \cap H$ is `good with respect
to $(T, Z)$' if it satisfies the conditions (G1) - (G3) above.
Let $\sP$ be a property of schemes. We shall say that $X \cap H$ satisfies $\sP$ for  
a general hypersurface $H \subset \P^n_k$ containing $Z$  
if for all $d \gg 0$, there is a nonempty open 
subscheme $\sU$ of the linear system $|H^0(\P^n_k, \sI_Z(d))|$
such that $X \cap H$ satisfies $\sP$ for all $H \in \sU(k)$.
\end{comment}

\vskip .3cm

\begin{notat}\label{notat:Hyp-1}
For the remainder of \S~\ref{sec:Prelim},
we fix an infinite field $k$, a subscheme $X \subset \P^n_k$ and a closed
subscheme $Z \subset \P^n_k$.
\end{notat}

Our goal is to prove a Bertini-regularity theorem for hypersurface sections
of $X$ containing $Z \cap X$.
We begin with the following special case which is
\cite[Corollary~3.4.14]{Flenner} when $\ov{X}$ is regular.
The general version is derived from a result of Seidenberg \cite{Seidenberg}
as follows.

\begin{lem}\label{lem:B-regularity}
  Let $d \ge 1$ be any integer. Then a general hypersurface section of $X$ of
degree $d$ is regular along $X_\reg$. 
\end{lem}
\begin{proof}
We can assume $X_\reg \neq \emptyset$. 
We can replace $X$ by $X_\reg$ which allows us to assume that $X$ is
regular. We can replace $\P^n_k$ by its $d$-uple Veronese embedding
which allows us to assume that $d =1$.
Since $X$ is reduced, the scheme-theoretic closure $\ov{X}$ is the
topological closure of  
$X$ with its reduced induced closed subscheme structure.
 
By \cite[Theorem~1]{Seidenberg}, we can find a dense open subscheme
$\sU$ of the linear system $|H^0(\P^n_k, \sO_{\P^n_k}(1))|$
such that for every $H \in \sU(k)$, the hypersurface section
$\ov{X} \cap H$ is regular at every point of $\ov{X}_\reg \cap H$.
We remark that Seidenberg assumes in the statement of his theorem that
$\ov{X}$ is irreducible, but never uses it in his proof.
He uses this extra condition only in the latter sections of his paper for a
different purpose.  
Since $X$ is open in $\ov{X}$, it follows that
$X \cap H$ is regular at every point of $X_\reg \cap H$.
\end{proof}

\begin{lem}\label{lem:Elem*}
Suppose in the situation of \notbref{notat:Hyp-1} that for all $d \gg 0$, 
there exists a dense open subscheme $\sU \subset 
|H^0(\P^n_k, \sI_{Z\cap \ov{X}}(d))|$ such that
$H \cap X$ satisfies a property $\sP$ for every $H\in \sU(k)$. Then
for all $d \gg 0$, there exists
a dense open subscheme $\sV \subset |H^0(\P^n_k, \sI_Z(d))|$
such that $H \cap X$ satisfies $\sP$ for every $H \in \sV(k)$.
\end{lem}
\begin{proof}
For every $d \ge 0$, we have  the canonical maps of coherent sheaves
\[
\sI_{Z}(d) \inj \sI_{Z \cap \ov{X}}(d) \surj {\sI_{Z \cap \ov{X}}}/{\sI_{\ov{X}}}(d)
\]
such that the composite map is surjective.
We let $d \gg 0$ be such that the induced map
$H^0(\P^n_k, \sI_Z(d)) \to H^0(\ov{X}, {\sI_{Z \cap \ov{X}}}/{\sI_{\ov{X}}}(d))$
is surjective.
%Call it $\rho$.
We then get a commutative diagram of rational maps between
the linear systems
\begin{equation}\label{eqn:Inf-closed-0}
\xymatrix@C.8pc{
|H^0(\P^n_k, \sI_Z(d))| \ar@{^{(}->}[rr]^-{\phi_1} \ar@{.>}[dr]_-{\phi_2}
& & |H^0(\P^n_k, \sI_{Z\cap \ov{X}}(d))| \ar@{.>}[dl]^-{\phi_3} \\
& |H^0(\ov{X}, {\sI_{Z \cap \ov{X}}}/{\sI_{\ov{X}}}(d))|,}
\end{equation}
where $\phi_2$ and $\phi_3$ are rational maps which are projections with
respect to some centers. They define smooth surjective
morphisms when restricted to the domains of definitions which are open dense subsets consisting of hypersurfaces not containing ${X}$. 

We now choose a dense open $\sU_1$ contained in the intersection of $\sU$ with the domain of definition of $\phi_3$ and let $\sU' = \phi^{-1}_{3}(\phi_3(\sU_1))$. Then $\sU'$ is open in $|H^0(\P^n_k, \sI_{Z\cap \ov{X}}(d))|$. It suffices to show that $\phi^{-1}_1(\sU')$ is nonempty.
% Let $\sV' = \phi^{-1}_{2}(\phi_3(\sU_1))$.
But this clear because
$\phi^{-1}_{2}(\phi_3(\sU_1))$ is nonempty and is contained in $\phi^{-1}_1(\sU')$.
%Then $\sV'$ is nonempty and $\sV'\subset \phi^{-1}_1(\sU')$.
\end{proof}

The following is the main result of this section.

\begin{prop}\label{prop:B-regularity-*}
  In the situation of \notbref{notat:Hyp-1}, 
assume further that $Z \cap \ov{X}$ is a reduced finite subscheme contained in $X_\reg$.
Then a general hypersurface $H \subset \P^n_k$ containing $Z$ has the property that
$X \cap H$ is regular along $X_\reg$.
\end{prop}
\begin{proof}
  As in \lemref{lem:B-regularity}, we can assume that $X \neq \emptyset$ is regular. Using our assumptions and \lemref{lem:Elem*}, we can assume that $Z \subset X_\reg$. The proposition is trivial if $\dim(X) = 0$. We shall therefore assume that $\dim(X) \ge 1$. Since $X$ is integral, $\ov{X}$ is the topological closure of $X$ with reduced closed subscheme structure. We
  let $V = \ov{X} \setminus Z$. We shall assume that $d$ is large enough so that $\sI_Z(d)$ is globally generated. 

Using \lemref{lem:B-regularity}, we can find for all $d \gg 0$, a dense open 
subscheme $\sU$ of the linear system $|H^0(\P^n_k, \sI_Z(d))|$
such that for every $H \in \sU(k)$, the hypersurface section
$X \cap H$ is regular along $V$.
We refer to the proof of \cite[Theorem~1]{KA} to see how this is achieved.
It remains to show that for every $x \in Z$ (note that $Z$ is finite)
and $d \gg 0$, there is a dense open subscheme $\sU_x \subset
|H^0(\P^n_k, \sI_Z(d))|$ such that for every $H \in \sU_x(k)$, the 
hypersurface section $X \cap H$ is regular at $x$. 

We write $Z = \{x\} \cup Z'$ with $x \notin Z'$ so that $\sI_Z = \sI_{\{x\}}\sI_{Z'}$.
This implies that 
\[
\frac{\sI_{\{x\}}}{\sI^2_{\{x\}} + \sI_Z} = 
\frac{\sI_{\{x\}}}{\sI_{\{x\}}(\sI_{\{x\}}+ \sI_{Z'})}
= \frac{\sI_{\{x\}}}{\sI_{\{x\}} \sO_{\P^n_k}} = 0,
\]
where the second equality occurs because $\{x\}$ and $Z'$ are disjoint.
On the other hand, the map
$H^0(\P^n_k, \sI_Z(d)) \otimes_k \sO_{\P^n_k} \to \sI_Z(d)$ is surjective
for $d \gg 0$. Hence the map of sheaves
\begin{equation}\label{eqn:B-regularity-*-1}
H^0(\P^n_k, \sI_Z(d)) \otimes_k \sO_{\{x\}} \to \sI_Z(d) 
\otimes_{\sO_{\P^n_k}} \sO_{\{x\}}; \ \ 
(s \otimes a \mapsto as)
\end{equation}
is surjective for $d \gg 0$.
We conclude that the canonical map  $H^0(\P^n_k, \sI_Z(d)) \otimes_k  \sO_{\{x\}} \to 
{\sI_{\{x\}}}/{\sI^2_{\{x\}}}(d)$ of sheaves is surjective for $d \gg 0$.

Let $\fm_x$ denote the maximal ideal of the local ring $\sO_{X,x}$. Since $Z$ is reduced, we have a surjection ${\sI_{\{x\}}}/{\sI^2_{\{x\}}} \surj {\fm_x}/{\fm^2_x}$. Since ${\sI_{\{x\}}}/{\sI^2_{\{x\}}}(d) \cong {\sI_{\{x\}}}/{\sI^2_{\{x\}}}$ and ${\fm_x}/{\fm^2_x} (d) \cong {\fm_x}/{\fm^2_x}$, it follows from the above that the canonical map of $k(x)$-vector spaces
\begin{equation}\label{eqn:B-regularity-*-2}
H^0(\P^n_k, \sI_Z(d)) \otimes_k k(x) \to {\fm_x}/{\fm^2_x}
\end{equation}
is surjective for $d \gg 0$.
But this implies from ~\eqref{eqn:B-regularity-*-1} that its restriction
$H^0(\P^n_k, \sI_Z(d))  \to {\fm_x}/{\fm^2_x}$ can not be zero.
Note that $\dim_{k(x)} {\fm_x}/{\fm^2_x} =  \dim(X) \ge 1$.
Let us call this restriction map $\psi_x$.

If we  choose any element $f \in H^0(\P^n_k, \sI_Z(d))$,
then $f$ will not die in ${\fm_x}/{\fm^2_x}$ under $\psi_x$ if and only if
the hypersurface $H_f$ contains $Z$ and $X \cap H_f$ is regular at $x$.
This uses our assumption that $x \in X_\reg$.
Since $|{\rm Ker}(\psi_x)|$ is a proper closed subscheme of
$|H^0(\P^n_k, \sI_Z(d))|$, 
it follows that there is a dense open subscheme
$\sU_x \subset |H^0(\P^n_k, \sI_Z(d))|$ for $d \gg 0$ such that
every $H \in \sU_x(k)$ contains $Z$ and $X \cap H$ is regular at $x$.
This finishes the proof.
\end{proof}

\begin{cor}\label{prop:B-regularity-spl}
  In the situation of \notbref{notat:Hyp-1}, assume further that $Z$ is a reduced finite
  closed subscheme contained in $X_\reg$. Then a general hypersurface $H \subset \P^n_k$ 
containing $Z$ has the property that $X \cap H$ is regular along $X_\reg$.
\end{cor}
\begin{proof}
  Under the assumption of the corollary, $Z$ is closed in $\ov{X}$ and hence
  \propref{prop:B-regularity-*} applies.
\end{proof}

\begin{remk}\label{remk:Failure}
  It is easy to see that the condition that $Z \cap \ov{X}$
  is reduced can not be relaxed in the above two results.
\end{remk}

We end this section with the following result, to be used in
\S~\ref{sec:Inf**}.
%Let $T \subset \P^n_k$ be a finite set of points containing $\Delta(X)$ such that $T \cap Z = \emptyset$. Note that $T$ may not be closed.

\begin{lem}\label{lem:good}
  In the situation of \notbref{notat:Hyp-1}, assume further that $T \subset \P^n_k$ is a
  finite set of (not necessarily closed) points containing $\Delta(X)$ such that
  $T \cap Z = \emptyset$. Then a general hypersurface $H \subset \P^n_k$ containing $Z$
  satisfies the following.
\begin{enumerate}
\item
$T \cap H = \emptyset$.
\item
$X \cap H$ is an effective Cartier divisor on $X$.
\end{enumerate}
\end{lem}
\begin{proof}
  We can assume $\dim(X) \ge 1$ because there is nothing to prove otherwise.
  We only need to prove (1) as (2) is its easy consequence.
  Since (1) is clearly an open condition on the points of $|H^0(\P^n_k, \sI_{Z}(d))|$,
  we only need to show that this linear system is nonempty for all $d \gg 0$.
  But this follows easily from \cite[Lemma~4.6]{GLL}.
\end{proof}

\section{Bertini for other properties over
infinite fields}\label{sec:Inf**}
In this section, we shall prove our remaining Bertini theorems over infinite fields.
The key step in the proof is the following algebraic lemma.

\subsection{The key lemma}\label{sec:Key}
Let $a, b \ge 0$ be two integers. Recall that a Noetherian scheme $X$ is called an
 $S_b$-scheme if for all points $x \in X$, one has
 ${\rm depth}(\sO_{X,x}) \ge {\rm min}(b, \dim(\sO_{X,x}))$
 (see \cite[pp.~183]{Matsumura}). 
We shall say that $X$ is an $(R_a+S_{b})$-scheme
if it is regular in codimension $a$ and is an $S_b$-scheme.
For integers $a, b \ge 0$, we let
\[
  X_{S_b} = \{x \in X|\sO_{X,x} \ \mbox{is \ an} \ S_b\mbox{-ring}\} \ \ \mbox{and}
\]
\[
X_{R_a + S_b} = \{x \in X|\sO_{X,x} \ \mbox{is \ an} \ (R_a + S_b)\mbox{-ring}\}.
\]

It follows from Auslander's theorem (see \cite[Proposition~6.11.2]{EGA-IV})
and openness of regular loci that $X_{S_b}$ and $X_{R_a + S_b}$ have
canonical structure of open subschemes of $X$.
We let $X_\nor \subset X$ be the set of points $x \in X$ such that
$\sO_{X,x}$ is integrally closed. Note that $X_\nor=X_{R_1 + S_2}$.
In particular, $X_\nor$ has the structure of an open subscheme of $X$. 
For $r \ge 0$, we let
\begin{equation}\label{eqn:SR-locus}
\Sigma^r_{X} = \left\{x \in X | \ \dim(\sO_{X,x}) \ge r+1 \ \mbox{and} \ 
  {\rm depth}(\sO_{X,x}) = r\right\}.
\end{equation}

We shall say that a scheme $X$ is
locally embeddable in a regular scheme if every point of $X$
has an affine neighborhood which is a closed subscheme of a 
Noetherian regular scheme (see \cite[Proposition~5.11.1]{EGA-IV}).

The key lemma for proving the Bertini theorems for $(R_a + S_b)$-schemes
over any field is the following result of independent
interest.

\begin{lem}\label{lem:Key*}
Let $X$ be a Noetherian separated scheme
of pure dimension $d \ge 0$. Assume that $X$ is locally embeddable in a regular
scheme. Assume also that $X$ is an $S_r$-scheme for some integer $r \ge 0$.
Then $\Sigma^r_{X}$ is a finite set.
\end{lem}
\begin{proof}
For any integer $0 \le m \le d$, let 
\[
W_m = \{x \in X|\ {\rm codepth}(\sO_{X,x}):= \dim(\sO_{X,x}) -
{\rm depth}(\sO_{X,x}) \ge m\}.
\]
Since $X$ is locally embeddable in a regular scheme, it follows from
Auslander's theorem (see \cite[Proposition~6.11.2]{EGA-IV}) that $W_m$
is closed in $X$. We let $W = \stackrel{d}{\underset{m = 1}\bigcup} W_m$.
Then $\Sigma^r_{X} \subset W$.

Suppose that $\Sigma^r_{X}$ is infinite. Then we can find $x_0 \in  \Sigma^r_{X}$
different from the generic points of $W_m$. Let $m_0 = {\rm codepth}(\sO_{X,x_0}) > 0$.
Let $x_1$ be the generic point of an irreducible component of $W_{m_0}$ passing
through $x_0$ and let $e = \dim(\sO_{X, x_0}) - \dim(\sO_{X,x_1}) > 0$.
As ${\rm codepth}(\sO_{X,x_1}) \ge m_0$, we get
\[
  {\rm depth}(\sO_{X,x_1}) \le {\rm depth}(\sO_{X,x_0}) -e = r-e < r.
\]
On the other hand, we have
\[
  \dim(\sO_{X,x_1}) = \dim(\sO_{X, x_0}) - e > r - e \ge {\rm depth}(\sO_{X,x_1}).
\]
Therefore, we get ${\rm depth}(\sO_{X,x_1}) < {\rm min}(r, \dim(\sO_{X,x_1}))$.
This contradicts our hypothesis that $X$ is an $S_r$-scheme.
\end{proof}

\begin{cor}\label{cor:Cat-exm}
  Let $X$ be a equidimensional Noetherian separated scheme.
Assume that $X$ is essentially of finite type over a Noetherian regular ring or over
  a complete Noetherian local ring. If $X$ is an $S_r$-scheme for some $r \ge 0$, then
  $\Sigma^r_X$ is a finite set.
%Then it satisfies the assertion of \lemref{lem:Key*}.
\end{cor}

\vskip .3cm

\subsection{Bertini when $X_\reg$ is not smooth}
\label{sec:B-reg*}
Our set-up for this subsection is the following.
\begin{notat}\label{notat:BT-10}
We let $k$ be an infinite field and $X \subset \P^n_k$ a subscheme of pure dimension $m$.
For $r \ge 0$, let $\Sigma^r_X$ be as in ~\eqref{eqn:SR-locus}.
Let $T \subset \P^n_k$ be a finite (not necessarily closed) subset containing $\Delta(X)$
(cf. \notbref{notat:Delta-X}).
We let $Z \subset \P^n_k$ be a closed subscheme such that
$Z \cap T = \emptyset$ and
$Z \cap \ov{X}$ is a reduced finite subscheme contained in $X_\reg$.
\end{notat}
We shall now prove Bertini theorems for $X$ under the above set-up.
Note that as we do not assume $X_\reg$ is smooth over $k$, these Bertini theorems are
applicable even if $k$ is an imperfect field. For instance, such results will play
a key role in the generalization of \cite[Thm.~4.1]{JS} to regular (but not smooth)
varieties over local fields.

% If $m \ge 1$, we let $\Sigma(X) = \stackrel{m-1}{\underset{r = 0}\bigcup} \Sigma^r_X$, where

\begin{thm}\label{thm:Bertini-reg-0}
  In the situation of \notbref{notat:BT-10}, a general hypersurface $H \subset \P^n_k$
  containing $Z$ satisfies the following.
\begin{enumerate}
\item
  $T \cap H = \emptyset$. In particular, $X\cap H$ is an effective Cartier divisor on $X$.
\item
  $X_\reg \cap H$ is regular.
\item
  $X_\red \cap H$ is reduced.
\item
  $X_{\nor} \cap H$ is normal.
\item
  $X_{R_a+S_b} \cap H$ is an $(R_a + S_b)$-scheme.
\end{enumerate}
\end{thm}
\begin{proof}
The properties (3) and (4) are special cases of (5) since
  reducedness is equivalent to $(R_0 + S_1)$ and normality is equivalent to
  $(R_1 + S_2)$. The properties (1) and (2) follow directly from
  \propref{prop:B-regularity-*} and \lemref{lem:good}. Since all properties are open
  conditions on the parameter space of hypersurfaces of a fixed degree,
  it remains to prove (5).

  Since $X$ is of pure dimension $m$ and $X_{R_a+S_b} \subset X$ is open,
  it follows that $X_{R_a+S_b}$ is of pure dimension at most $m$.
  We can therefore assume that $X$ is an
  $(R_a + S_b)$-scheme of dimension $m \ge 1$.
  We can also assume by \lemref{lem:Elem*} that $Z \subset \ov{X}$.

Since $Z \subset X_\reg$, it follows that $Z \cap \Sigma^b_X = \emptyset$.
We now apply \propref{prop:B-regularity-*} and \lemref{lem:good}, where we replace
the set $T$ in the latter result by current $T \cup
\Sigma^b_X$. It follows that for all $d \gg 0$, there is a nonempty
open subscheme $\sU \subset |H^0(\P^n_k, \sI_{Z}(d))|$ such that 
every $H \in \sU(k)$ satisfies (2) and the property that
$X \cap H \cap \Sigma^b_X = \emptyset$.

We let $H \in \sU(k)$ and $Y = X \cap H$. We show that $Y$  is an
$(R_a + S_b)$-scheme. Since $X$ has pure dimension $m$ and
is regular in codimension $a$, it follows that $\dim(X_{\rm sing}) \le m-a-1$.
Since $Y$ does not contain any generic point of $X_\sing$,
it follows that $\dim(Y \cap X_{\rm sing}) \le m-a-2$.
Since $Y$ is catenary, it follows that
$(Y\cap X_{\rm sing})$ has codimension at least $(m-1) - (m-a-2) = a+1$ in 
$Y$. Since $Y \cap X_\reg$ is regular, 
we conclude that $Y$ is regular in codimension $a$.
We shall now show that $Y$ is an $S_b$-scheme.

We fix a point $y \in Y$. If $y \in X_{\reg}$, then
$\sO_{Y,y}$ is regular and hence an $S_b$-ring. 
We can therefore assume that $y \in X_{\rm sing}$.
Since $y \notin \Sigma^b_X$, 
we see that either $\dim(\sO_{X,y}) \le b$ or
${\rm depth}(\sO_{X,y}) \ge b+1$. 
If $\dim(\sO_{X,y}) \le b$, then we must have ${\rm depth}(\sO_{X,y})
= \dim(\sO_{X,y})$ since $X$ is an $S_b$-scheme. Equivalently,
$\sO_{X,y}$ is Cohen-Macaulay. 
Since $Y$ is an effective Cartier divisor on $X$, it follows that
$\sO_{Y,y}$ is also Cohen-Macaulay. 

If $\dim(\sO_{X,y}) \ge b+1$, then
we just saw that ${\rm depth}(\sO_{X,y}) \ge b+1$. In this case, it is an elementary
fact that ${\rm depth}(\sO_{Y,y}) \ge b$
(see \cite[Proposition~16.4.6(ii)]{EGA}).
We have therefore shown that $\sO_{Y,y}$ is either Cohen-Macaulay or
${\rm depth}(\sO_{Y,y}) \ge b$ for every
$y \in Y$. But this is equivalent to saying that
$Y$ is an $S_b$-scheme.
This finishes the proof.
\end{proof}

\subsection{Bertini for irreducibility and integrality}
\label{sec:Irr-inf}
We need the following additional input in order to prove Bertini theorems for
irreducibility and integrality.

\begin{lem}\label{lem:B-irr-inf}
 Let $k$ be an infinite field and $Y \subset \P^n_k$ an irreducible subscheme of
  dimension $m \ge 2$. Let $W \subset \P^n_k$ be a closed subscheme such that
the codimension of $W \cap Y$ is at least two in $Y$.
Then a general hypersurface $H \subset \P^n_k$ containing $W$
has the property that $Y \cap H$ is irreducible.
\end{lem}
\begin{proof}
Let $\ov{k}$ be an algebraic closure of $k$. We let $Y_1, \cdots , Y_r$ denote the irreducible components of $Y_{\ov{k}}$.  
It follows from \cite[Theorem~1]{KA} that for all $d \gg 0$, there is
a dense open subscheme $\sU_1 \subset |H^0(\P^n_{\ov{k}}, \sI_{W_{\ov{k}}}(d))|$
such that every $H \in \sU_1(\ov{k})$ has the property that
$H \cap Y_i$ is irreducible for each $i$.
We let $V = H^0(\P^n_k, \sI_{W}(d))$. 

Let $K/k$ be a finite extension over which $\sU_1$ and $Y_i$'s are defined.
Let $\pi \colon \P(V)_K \to \P(V)$ be the projection. Since $\pi$ is a finite
morphism, we see that $\pi(\P(V)_K \setminus \sU_1)$ is a proper closed subset of
$\P(V)$. Let $\sU$ be its complement in $\P(V)$. Then $\pi^{-1}(\sU) = \sU_K$ is
open dense in $\sU_1$ and is defined over $k$. Let $H \in \sU(k)$. As
$Y_K \to Y$ is finite and flat, its restriction to each $Y_i$ surjects onto $Y$.
In particular, $Y_i \cap H_K \to Y \cap H$ is surjective. This forces 
$Y \cap H$ to be irreducible.
\end{proof}

%We now let $X, T$ and $Z$ be as described in the beginning of \S~\ref{sec:B-reg*}.

\begin{thm}\label{thm:B-irreducible-inf}
  In the situation of \notbref{notat:BT-10}, assume further that $X$ is irreducible of
  dimension $m \ge 2$. 
Then a general hypersurface $H \subset \P^n_k$ containing $Z$ satisfies the following.
\begin{enumerate}
\item
  $T \cap H = \emptyset$.
\item
$X_\reg \cap H$ is regular.
 \item
   $X \cap H$ is irreducible.
 \item
 $X_\red \cap H$ is integral.
\end{enumerate}
\end{thm}
\begin{proof}
Combine \thmref{thm:Bertini-reg-0} with \lemref{lem:B-irr-inf}.
\end{proof}

\subsection{Bertini when $X_\reg$ is smooth}\label{sec:B-sm*}
Let $k$ be any field and $Y$ a $k$-scheme. Recall that for a point $y \in Y$, the embedding
dimension of $Y$ at $y$ is $\edim_y(Y) =
\dim_{k(y)}(\Omega^1_{Y/k}\otimes _{\sO_Y} k(y))$,
where $\Omega^1_{Y/k}$ is the Zariski sheaf of K{\"a}hler differentials on $Y$.
For any integer $e \ge 0$, we let $Y_e$ denote the subscheme
of points $y \in Y$ such that ${\rm edim}_y(Y) = e$.
Since $y \mapsto {\rm edim}_y(Y)$ is an upper semi-continuous function 
on $Y$ (see \cite[Example~III.12.7.2]{Hartshorne}), 
it follows that $Y_e$ is a locally closed subscheme of $Y$. 
We let $\tdim(Y) = {\underset{e \ge 0}{\rm max}}
\{e + \dim(Y_e)\}$.

\begin{notat}\label{notat:BT-11}
  Let $k$ be an infinite field. Let $X \subset \P^n_k$ be a subscheme of
pure dimension $m$ and $Z \subset \P^n_k$ a closed subscheme.
Let $T \subset \P^n_k$ be a finite (not necessarily closed)
subset containing $\Delta(X)$ (cf. \notbref{notat:Delta-X}). Assume that $X_\reg$ is
smooth and $Z \subset \P^n_k$  satisfies the following.
%\begin{equation}\label{eqn:Extra-condn}
\[
  (1) \ Z \cap T = \emptyset, \ (2) \ Z \cap X_\reg \ \mbox{is \ dense in} \
  Z \cap \ov{X} \ \ \mbox{and} \ \ (3) \ \tdim(Z \cap X_\reg) < \dim(X_\reg).
\]
%\end{equation}
Note that (2) and (3) are much weaker than the condition we imposed on $Z$ in
\S~\ref{sec:B-reg*}.
\end{notat}

Under (3), we have the following result from \cite[Theorem~7]{KA}.  

\begin{lem}\label{lem:KL-smooth}
 A general hypersurface $H \subset \P^n_k$ containing $Z$ has the property that $X_\reg \cap H$ is smooth.
\end{lem}

We shall now prove Bertini theorems over $k$ by
weakening the condition on $Z \cap X$ (allowing it to have positive dimension)
on the one hand, and strengthening the condition on $X_\reg$
(requiring it to be smooth) on the other hand.

\vskip .2cm

%The Bertini theorem for $(R_a+ S_b)$-property is the following.

\begin{thm}\label{thm:B-smooth-inf}
  In the situation of \notbref{notat:BT-11}, a general hypersurface $H \subset \P^n_k$
  containing $Z$ satisfies the following.
\begin{enumerate}
\item
  $T \cap H = \emptyset$.
  \item
$X_\reg \cap H$ is smooth.
\item
  $X_{R_a+ S_b} \cap H$ is an $(R_a + S_b)$-scheme if $Z \cap X \subset X_{S_{b+1}}$.
\item
  $X \cap H$ is reduced if $X$ is reduced and $Z \cap X \subset X_{S_2}$.
\item
  $X_\nor \cap H$ is normal if $Z \cap X \subset X_{S_3}$. 
\end{enumerate}
\end{thm}
\begin{proof}
  Under our assumptions,  (1) follows directly from \lemref{lem:good} while (2)
  follows from \lemref{lem:KL-smooth}. The statements (4) and (5) are special cases of
  (3). Hence, it remains to show (3). 

  Since $X$ is of pure dimension $m$ and
  $X_{R_a+ S_b}$ is open in $X$, we can assume without loss of generality that
  $X$ is an $(R_a + S_b)$-scheme. Since $Z \cap X \subset  X_{S_{b+1}}$, it follows
  that $Z \cap \Sigma^b_X = \emptyset$. We now apply \lemref{lem:KL-smooth}, and
  \lemref{lem:good} (with $T$ replaced by $T \cup \Sigma^b_X$). We get that for
  all $d \gg 0$, there is a nonempty open subscheme $\sU \subset |H^0(\P^n_k, \sI_{Z}(d))|$
  such that for every $H \in \sU(k)$, $X_\reg \cap H$ is smooth over $k$ and
  $(X \cap H) \cap \Sigma^b_X = \emptyset$. We now repeat the argument in the
  last three paragraphs in the proof of \thmref{thm:Bertini-reg-0} to conclude that
  $X\cap H$ satisfies (3). 
\end{proof}

\begin{thm}\label{thm:B-smooth-irr-inf}
  In the situation of \notbref{notat:BT-11}, assume further that $X$ is irreducible of
  dimension $m \ge 2$. Assume also that the codimension of $Z \cap X$ is at least two in $X$.
  Then a general hypersurface $H \subset \P^n_k$ containing $Z$ satisfies the following.
  \begin{enumerate}
  \item
    $T \cap H = \emptyset$.
  \item
$X_\reg \cap H$ is smooth.  
\item
$X \cap H$ is irreducible.
\item
$X \cap H$ is integral if $X$ is integral and $Z \cap X \subset X_{S_2}$.
\end{enumerate}
\end{thm}
\begin{proof}
  Combine \thmref{thm:B-smooth-inf} with \lemref{lem:B-irr-inf}. 
\end{proof}

\section{Bertini for the $(R_a + S_b)$-property over finite fields}
\label{sec:BNT*}
In this section, we shall use \lemref{lem:Key*} and the Bertini-smoothness theorems
of Poonen \cite{Poonen-1} and Wutz \cite{Wutz} to prove our Bertini theorem for the
$(R_a + S_b)$-property over finite fields. 
We shall assume throughout this section that $k$ is a finite field.

\subsection{Poonen's density function}\label{sec:DF}
Let $S = k[x_0, \ldots , x_n]$ be the homogeneous coordinate ring of $\P^n_k$.
Let $S_d \subset S$ be the $k$-subspace of homogeneous polynomials of degree
$d \ge 0$ and let $S_{\homg} = {\underset{d \ge 1}\bigcup} S_d
\subset S$. For each $f \in S_d$, let $H_f \subset \P^n_k$ denote the
hypersurface of $\P^n_k$ defined by $f$. 
Let $Z \subset \P^n_k$ be a closed subscheme defined by the sheaf of 
ideals $\sI_Z$. For any integer $d \ge 0$,
let $I_{Z,d} = H^0(\P^n_k, \sI_Z(d))$. Then
$I_Z = {\underset{d \ge 0}\oplus} I_{Z,d}$ is the homogeneous ideal
of $S$ which defines $Z$. We let $I^Z_{\homg} = {\underset{d \ge 1}\bigcup} I_{Z,d}$.

\begin{comment}
In particular, $S_Z := S/{I_Z}$ is the homogeneous coordinate ring of $Z$. 
We let $\wt{S}_Z = {\underset{i \ge 0}\oplus} H^0(Z, \sO_Z(d))$.
The short exact sequence
\begin{equation}\label{eqn:ses*}
0 \to \sI_Z(d) \to \sO_{\P^n_k}(d) \to \sO_{Z}(d) \to 0
\end{equation}
shows that there is a canonical inclusion of graded $k$-algebras
$S_Z \inj \wt{S}_Z$. This inclusion is an isomorphism in all sufficiently high
degrees.
\end{comment}

Recall that there is a surjective
map of Zariski sheaves $\sO^{\oplus (n+1)}_{\P^n_k} \surj \sO_{\P^n_k}(1)$
on $\P^n_k$, given by $(f_0, \ldots , f_n) \mapsto 
\stackrel{n}{\underset{i = 0}\sum} f_ix_i$.
Tensoring this surjection with $\sI_Z$ and twisting by $d \gg 1$,
we see that there exists $c_Z \gg 1$ such that
\begin{equation}\label{eqn:surjective}
I_{z,d+1} = S_1I_{z,d} \ \mbox{for \ all} \ d \ge c_Z.
\end{equation}

As the referee has pointed out, $c_Z$ is closely related to the Castelnuovo-Mumford regularity
of $\sI_Z$.

For a subset $\sP \subset I^Z_{\homg}$, we let
%\begin{equation}\label{eqn:density-0}
$\mu_Z(\sP) = {\underset{d \to \infty}{\rm lim}} \frac{\#(\sP \cap I_{Z,d})}
{\# I_{Z,d}}$,
%\end{equation}
if the limit exists.
We call $\mu_Z(-)$ {\sl the density function} on the power set of
$I^Z_{\homg}$. When we replace the limit on the right hand side 
in the definition of $\mu_Z(\sP)$
%~\eqref{eqn:density-0}
by the limit superior (resp. inferior), we shall denote the associated
density function by $\ov{\mu}_Z(-)$ (resp. $\un{\mu}_Z(-)$).
When $Z = \emptyset$, we shall write $\mu_Z(\sP)$ as $\mu(\sP)$.
%It is clear that in this case, we have $I_{Z,d} = S_d$ for every $d \ge 0$.
The following is easy to prove.

\begin{lem}\label{lem:Elem}
Given two subsets $\sP$ and $\sP'$ of $I^Z_{\homg}$, the following hold.
\begin{enumerate}
\item
If $\mu_Z(\sP)$ exists, then one has $0 \le \mu_Z(\sP) \le 1$.
Moreover, $\mu_Z(\sP) > 0$ implies that there exists $d_0 \gg 0$ such that
$(\sP \cap I_{Z,d}) \neq \emptyset$ for all $d \ge d_0$.
\item
If $\mu_Z(\sP)$ and $\mu_Z(\sP')$ both exist and $\mu_Z(\sP') = 1$, then
$\mu_Z(\sP \cap \sP') = \mu_Z(\sP)$.
\end{enumerate} 
\end{lem}

\subsection{Main results over finite fields}\label{sec:BNT}
In this subsection, we shall prove the Bertini theorem
for the $(R_a + S_b)$-property over the finite field $k$. 
This implies Bertini theorems for normality
and reducedness. In particular,  we extend Seidenberg's
Lefschetz hyperplane section theorem \cite{Seidenberg}
to normal varieties over finite fields.
As much as we are aware, these results were not known over finite fields in
any form before.

\vskip .3cm

We need the following generalization of \cite[Lemma~3.1]{Poonen-3}.

\begin{lem}\label{lem:positive-dim}
Let $W$ be a positive dimensional irreducible subscheme of $\P^n_k$
(resp. $\P^n_{\ov{k}}$). Let $Z$ be a closed subscheme of $\P^n_k$ such that
$W_\red \not\subset Z$ (resp. $W_\red \not\subset Z_{\ov{k}}$).   
Let $\sP = \{f \in I^Z_{\homg}| W \subset H_f\}$.
Then $\mu_Z(\sP) = 0$.
\end{lem}
\begin{proof}
We can assume $W$ to be reduced to prove the lemma.
We can then replace $W$ by $\ov{W}$ to prove the lemma. We can therefore
assume that $W$ is closed in the ambient projective space.
If $W$ lies in $\P^n_{\ov{k}}$, then it lies
in the projective space defined over a finite field extension
of $k$. Since $H_f$ is defined over $k$, we can therefore replace $W$ by its 
scheme-theoretic image in $\P^n_k$ in order to prove the lemma.
In conclusion, we can assume that $W \subset \P^n_k$.

Since $W$ is irreducible and not contained in $Z$, it follows that
$U := W \setminus Z$ is a positive dimensional irreducible subscheme
of $\P^n_k$ disjoint from $Z$. Moreover, $W \subset H_f$ if and only if
$U \subset H_f$ for any $f \in I^Z_{\homg}$.
We can therefore assume that $W \cap Z = \emptyset$.

For any closed point $w \in W$, we know that $\ov{\mu}_Z(\sP)$ is bounded above by
\[
  \begin{array}{lll}
    \mu_Z(\{f \in  I^Z_{\homg}| w \in H_f\}) & = &
 \mu_Z(\{f \in  I^Z_{\homg}| \{w\} \cap H_f \neq \emptyset \}) \\
                                             & = & 1- (1-(\# k(w))^{-1}) \\
                                             & = & (\# k(w))^{-1},
  \end{array}
\]
where the second equality is given by
%$\mu_Z(\{f \in  I^Z_{\homg}| w \in H_f\}) = (\# k(w))^{-1}$ by
\lemref{lem:Zero-density}. Since $\dim(W) > 0$, we can choose
$(\# k(w))^{-1}$ to be arbitrarily small and this implies that
$\ov{\mu}_Z(\sP) = 0$. In particular, we get $\mu_Z(\sP) = 0$. 
\end{proof}

\begin{lem}\label{lem:Zero-density}
Let $Z$ and $W$ be two closed subschemes of $\P^n_k$ such that
$W$ is finite and $Z \cap W = \emptyset$.   
Let $\sP = \{f \in I^Z_{\homg}| W \cap H_f = \emptyset\}$.
Then $\mu_Z(\sP) = {\underset{w \in W}\prod} (1 - (\# k(w))^{-1})$.
\end{lem}
\begin{proof}
We can assume that $W$ is reduced in order to prove the lemma.
Since $Z \cap W = \emptyset$, we get $\sI_Z \cdot \sO_W = \sO_W$.
Since we also have $\sO_W(d) \cong \sO_W$, we conclude from
\cite[Lemma~2.1]{Poonen-2} that the map
$\phi_d \colon I_{Z,d} \to H^0(W, \sO_W)$, induced by
the map of sheaves $\sI_Z \to \sI_Z \cdot \sO_W$, is surjective
for $d \ge c_Z + \dim_k H^0(W, \sO_W)$, where $c_Z$ is as in 
~\eqref{eqn:surjective}. Since $H^0(W, \sO_W) \cong
{\underset{w \in W}\prod} k(w)$ and $\phi_d(f) = f|_W$ (see \S~\ref{sec:B-Sing*} for the definition of $f|_W$), it follows that
$\sP \cap I_{Z,d} = \phi^{-1}_d({\underset{w \in W}\prod} k(w)^{\times})$.
We therefore get
\[
\mu_Z(\sP) = {\underset{d \to \infty}{\rm lim}} 
\frac{\#(\phi^{-1}_d({\underset{w \in W}\prod} k(w)^{\times})}
{\# I_{Z,d}} = 
{\underset{d \to \infty}{\rm lim}} 
\frac{{\underset{w \in W}\prod} k(w)^{\times}}
{{\underset{w \in W}\prod} k(w)}
= {\underset{w \in W}\prod} (1 - (\# k(w))^{-1}).
\]
This finishes the proof.
\end{proof}

%\vskip .3cm

Our setting for the Bertini theorem for the $(R_a + S_b)$-property is as follows.
\begin{notat}\label{notat:BT-12}
  Let $k$ be a finite field and $X \subset \P^n_k$ a subscheme of pure dimension $m \ge 0$.
For $r \ge 0$, let $\Sigma^r_X \subset X$ be as in ~\eqref{eqn:SR-locus}.
Let $T \subset \P^n_k$ be a finite set containing $\Delta(X)$ (cf. \notbref{notat:Delta-X}).
Let $Z \subset \P^n_k$ be a closed subscheme such that $Z \cap T = \emptyset$ and
$\tdim(Z \cap X_{\rm sm}) < m$ (cf. \S~\ref{sec:B-sm*}).
\end{notat}

\begin{thm}\label{thm:R-S-main}
  In the situation of \notbref{notat:BT-12}, assume further that $X$ is an
  $(R_a+S_b)$-scheme for some integers $a, b \ge 0$ such that $Z \cap \Sigma^b_X = \emptyset$.
  Let $\sP \subset I^Z_{\homg}$ be the set of homogeneous polynomials $f$ such that the
  subscheme $X \cap H_f$ satisfies the following.
\begin{enumerate}
\item
$T \cap H_f = \emptyset$.
\item
$X_\sm \cap H_f$ is smooth.
\item
$X \cap H_f$ is an $(R_a+S_b)$-scheme of pure dimension $m-1$.
\end{enumerate}

Then $\sP$ contains a subset $\sP'$ such that $\mu_Z(\sP') > 0$.
\end{thm}
\begin{proof}
Since $k$ is perfect and $X$ is generically reduced.
it follows that $X_{\rm sm} = X_{\rm reg} \subset X$ is a dense
open in $X$. In particular, $\dim(X_{\rm sing}) \le m-1$.

We set $W = T \bigcup \Sigma^b_X$. 
We further write $W = W_1 \amalg W_2$, where
$W_1$ consists of the closed points of $\P^n_k$ lying in $W$ and $W_2$ 
consists of the non-closed points of $\P^n_k$ lying in $W$.
We write $W_1 = \{P_1, \ldots , P_{r}\}$.

Consider $W_1 = \stackrel{r}{\underset{i =1}\amalg} \Spec(k(P_i))$ as a finite closed 
subscheme of $\P^n_k$ and let $W'_1 =  \stackrel{r}{\underset{i =1}\prod}
k(P_i)^\times \subset H^0(W_1, \sO_{W_1})$.
We let  $U = X_{\rm sm} \setminus W_1$ and $V = U \cap Z$. Letting 
\[
\sP_0 = \left\{f \in I^Z_{\homg} | H_f \cap U \ \mbox{is \ smooth \
of \ dimension} \ m-1, \ \mbox{and} \ f|_{W_1} \in\, W'_1\right\},
\]
it follows from \propref{prop:Poonen} (see \remref{remk:WP}) that
\begin{equation}\label{eqn:RS-main-1-0}
\mu_Z(\sP_0) = \frac{\# W'_1}{\# H^0(W_1, \sO_{W_1})}
\frac{\zeta_V(m+1)}{\zeta_U(m+1) \stackrel{m-1}{\underset{e = 0}\prod}
\zeta_{V_e}(m-e)} > 0.
\end{equation}

For any point $x \in \P^n_k$, we let
$\sP_x = \{f \in I^Z_{\homg}| x \notin H_f\}$ and define
$\sP' = ({\underset{x \in W_2}\cap} \sP_x) \bigcap \sP_0$.
Since no point of $W_2$ is either closed in $\P^n_k$ or lies in $Z$,
it follows from \lemref{lem:positive-dim} that 
$\mu_Z(\sP_x) = 1$ for all $x \in W_2$. We conclude from
~\eqref{eqn:RS-main-1-0} and \lemref{lem:Elem} that
$\mu_Z(\sP') > 0$.

We shall now show that $\sP' \subseteq \sP$, which will finish the proof of the
theorem.
We fix an element $f \in \sP'$ and let $Y = X \cap H_f$. The property (1) is clear. 
It is clear from the definition of $\sP_0$ that $Y \cap X_{\rm sm}$
is smooth except possibly at the points of $W_1$.
However, it follows from the definition of $W'_1$ and $\sP_0$ that
$H_f$ contains no point of $W_1$. We conclude that $Y \cap X_{\rm sm}$ is smooth.
It remains to show that $Y$ is an $(R_a + S_b)$-scheme.

First, the argument given in the proof of \thmref{thm:Bertini-reg-0} shows
verbatim that $Y$ is regular in codimension $a$. Second, to show that $Y$ is an
$S_b$-scheme, we only need to show using (2) that $\sO_{Y,y}$ is an $S_b$-ring for
$y \in Y \cap X_\sing$. We now fix such a point. Since
$f \in {\underset{x \in W_2}\bigcap} \sP_x$ and also
$f \in k(P_i)^{\times}$ for every $1 \le i \le r$,
it follows that $y \notin W$. In particular, $y \notin \Sigma^b_X$.
But this implies that either $\dim(\sO_{X,y}) \le b$ or
${\rm depth}(\sO_{X,y}) \ge b+1$. We can now repeat the last two paragraphs
of the proof of \thmref{thm:Bertini-reg-0} to conclude the proof.
\end{proof}

\begin{remk}\label{remk:WP}
  We note here that we applied \propref{prop:Poonen} in the previous result
  with  $t =1$, and took
$X_\sm, W_1$ and $W'_1$ for $X_1, Y$ and $T$ of the proposition, respectively.
In this special case, \propref{prop:Poonen} is already known and
is due to Wutz \cite[Theorem~1.1]{Wutz} (and to Poonen
\cite[Theorem~1.2]{Poonen-1} when $Z = \emptyset$). 
The reference to \propref{prop:Poonen} was
only to recall the underlying notations.
\end{remk}

\begin{cor}\label{cor:B-reduced-fin}
  In the situation of \notbref{notat:BT-12}, assume further that $X$ is reduced and
  $Z \cap X \subset X_{S_2}$ (resp. normal and $Z \cap X \subset X_{S_3}$).
Let $\sP$ be the set of polynomials $f \in I^Z_\homg$ for which the following hold.
\begin{enumerate}
\item
  $T \cap H_f = \emptyset$.
  \item
$X_\sm \cap H_f$ is smooth.
\item
$X \cap H_f$ is reduced (resp. normal).
\end{enumerate}

Then $\sP$ contains a subset with positive density.
\end{cor}
\begin{proof}
Apply \thmref{thm:R-S-main} with $(a,b) = (0,1)$ (resp. $(a,b) = (1,2)$).
\end{proof}

\subsection{A Bertini theorem for singular schemes over finite fields}
\label{sec:B-Sing*}
In this subsection, we shall prove a Bertini theorem for singular schemes over the finite
field $k$.
This result will play a key role in the proofs of Theorems~\ref{thm:B-base-reg}
and ~\ref{thm:B-base-closed}.
We shall deduce this Bertini theorem as a consequence of an extension of
the main result of \cite{Wutz}.

Recall that the arithmetic zeta function for $X \in \Sch_k$
is defined to be the power series

\[
\zeta_X(t) = Z_X(q^{-t}) :=\exp\left(\stackrel{\infty}{\underset{s = 1}\sum} 
\frac{\# X(\F_{q^s})}{s} q^{-ts}\right) = {\underset{x \in X_{(0)}}\prod}
(1 - q^{-t\deg(x)})^{-1},
\]

where $X_{(0)}$ is the set of closed points of $X$. It is a consequence of Galois theory that $\zeta_X(t) \in \Z[[t]]$.
Furthermore, it was shown by Dwork that $Z_X(t) \in \Q(t)$.
If $X = Y \cup U$, where $Y \subset X$ is closed and $U = X \setminus Y$,
then $\zeta_X(t) = \zeta_Y(t)\zeta_U(t)$.

For a finite closed subscheme
$W \subset \P^n_k$ and a homogeneous polynomial $f \in S_d$,
we let $f|_W$ be the element of $H^0(W, \sO_W)$ that on each connected
component $W_i$ of $W$ equals the restriction of $x^{-d}_jf$ to $W_i$, where
$j = j(i)$ is the smallest $j \in \{0, 1, \ldots , n\}$ such that the
coordinate $x_j$ is invertible on $W_i$. 
We refer the reader to \S~\ref{sec:B-sm*} for the definitions of $X_e$ and
$\tdim(X)$.
%We shall follow the notations of \notbref{notat:Notn}.
By a smooth scheme of dimension $m$, we shall mean a smooth
scheme of pure dimension $m$.

Let $Y \subset \P^n_k$ be a finite closed subscheme and
$T \subset H^0(Y, \sO_Y)$ a nonempty subset.
Let $U_1, \ldots , U_t$ be smooth and pairwise disjoint
equidimensional
subschemes of $\P^n_k$ such that $Y \cap U_i = \emptyset$ for every $i$.
Let $Z \subset \P^n_k$ be a closed subscheme such that $Y \cap Z =
\emptyset$ and $\tdim(Z \cap U_i) < m_i$ for $1 \le i \le t$ (cf. \S~\ref{sec:B-sm*}),
where $m_i = \dim(U_i)$. Let $V_i = U_i \cap Z$. Define
\begin{equation}\label{eqn:PW-gen-0}
\sP = \{f \in I^Z_\homg|f|_Y \in T \ \mbox{and} \ U_i \cap H_f \ 
\mbox{is \ smooth \ of \ dimension} \ m_{i} - 1 \ \forall \ 1 \le i \le t\}.
\end{equation}

%\enlargethispage{20pt}

The following result is obtained by mimicking the proof (but not a consequence)
of the Bertini theorem of Wutz \cite{Wutz}.

\begin{prop}\label{prop:Poonen}
Under the above assumptions, we have 
\[
\mu_Z(\sP) = \frac{\# T}{\# H^0(Y, \sO_Y)} \stackrel{t}{\underset{i = 1}\prod}
\frac{1}{\zeta_{U_i \setminus V_i}(m_i +1) 
\stackrel{m_i -1}{\underset{e = 0}\prod}\zeta_{(V_i)_e}(m_i - e)} > 0.
\]
\end{prop}
\begin{proof}
  This proposition would have been a direct consequence of \cite[Theorem~1.1]{Wutz}
  if the scheme $U := \ \stackrel{t}{\underset{i =1}\bigcup} U_i$
  was smooth and equidimensional
  (in particular, $m_i = m_j$ for all $1 \le i, j \le t$). Nonetheless, the general case
  is deduced by rewriting the argument of op. cit. appropriately.
  We give a sketch.

  For $U \in \Sch_k$, let $U_{< r}$ (resp. $U_{> r}$)
be the set of closed points
of $U$ of degree $< r$ (resp. $> r$). 
Let $c_Z$ be the integer found in ~\eqref{eqn:surjective}.
We define the following sets.
\begin{equation}\label{eqn:PW-gen-1}
\sP_r = \{f \in I^Z_{\homg}|f|_Y  \in T \ \mbox{and} \  U_i \cap H_f \ 
\mbox{is \ smooth \ of \ dimension $m_i -1$ \ \ at} \ x \ \forall \ i
\end{equation}
\[
\hspace*{9cm} \ 
\mbox{and} \ \forall \ 
\mbox{closed point} \ x \in (U_i)_{<r}\}.
\]
\[
\sQ^{\rm med}_{i,r} = {\underset{d \ge 1}\bigcup} 
\{f \in I_{Z,d}| \exists \ \mbox{a \ closed \ point} \ x \in U_i \ 
\mbox{with} \ r \le \deg(x) \le \frac{d-c_Z}{m_i + 1} \ \mbox{such \ that}
\]
\[
\hspace*{5.5cm} U_i \cap H_f \ \mbox{is \ not \ 
smooth \ of \ dimension $m_i -1$ \ at} \ x\}.
\]
\[
\sQ^{\rm high}_{U_i \setminus V_i} =
{\underset{d \ge 1}\bigcup} 
\{f \in I_{Z,d}| \exists \ \mbox{a \ closed \ point} \ x \in 
(U_i \setminus V_i)_{> {(d-c_Z)}/{(m_i +1)}}
\ \mbox{such \ that}  \  U_i \cap H_f
\]
\[
\hspace*{6cm} \ \mbox{is \ not \ 
smooth \ of \ dimension $m_i -1$ \ at} \ x\}.
\]
\[
\sQ^{\rm high}_{V_i} =
{\underset{d \ge 1}\bigcup} 
\{f \in I_{Z,d}| \exists \ \mbox{a \ closed \ point} \ x \in 
(V_i)_{> {(d-c_Z)}/{(m_i +1)}}
\ \mbox{such \ that}  \  U_i \cap H_f
\]
\[
\hspace*{6cm} \ \mbox{is \ not \ 
smooth \ of \ dimension $m_i -1$ \ at} \ x\}.
\]

Since $\{U_i\}_{1 \le i \le t}$ are pairwise disjoint, an easy 
calculation (e.g., see \cite[Lemmas~2.3, 2.4]{Wutz}) shows that
\[
\mu_Z(\sP_r) = \frac{\# T}{\# H^0(Y, \sO_Y)} 
\stackrel{t}{\underset{i = 1}\prod}
\left[{\underset{x \in (U_i \setminus V_i)_{< r}}\prod} (1 - q^{-(m_i + 1)\deg(x)})
\cdot {\underset{x \in (V_i)_{< r}}\prod} (1 - q^{-(m_i - e)\deg(x)}).
\right]
\]

On the other hand, it follows from \cite[Lemma~3.2]{Wutz} that
${\underset{r \to \infty}\lim} \ov{\mu}_Z(\sQ^{\rm med}_{i,r}) = 0$
for each $i$. Moreover, it follows from \cite[Lemmas~4.1, 4.2]{Wutz}
that $\ov{\mu}_Z(\sQ^{\rm high}_{U_i \setminus V_i}) = 0 = 
\ov{\mu}_Z(\sQ^{\rm high}_{V_i})$ for each $i$.
Since $\sP \subset \sP_r \subset \sP \bigcup \left(\stackrel{t}{\underset{i = 1}\cup}
(\sQ^{\rm med}_{i,r} \cup \sQ^{\rm high}_{U_i \setminus V_i} \cup \sQ^{\rm high}_{V_i})
\right)$, it follows that $\ov{\mu}_Z(\sP)$ and $\un{\mu}_Z(\sP)$ differ from
$\mu_Z(\sP_r)$ at most by $\stackrel{t}{\underset{i = 1}\sum} 
(\ov{\mu}_Z(\sQ^{\rm med}_{i,r}) + \ov{\mu}_Z(\sQ^{\rm high}_{U_i \setminus V_i}) + 
\ov{\mu}_Z(\sQ^{\rm high}_{V_i}))$.
We conclude that 
${\mu}_Z(\sP) = {\underset{r \to \infty}\lim}\mu_Z(\sP_r)$.

Since $\tdim(V_i) < m_i$ by our assumption, 
we get $\dim((V_i)_e) < m_i - e$ for every $1 \le i \le t$. 
We also have the inequality $\dim(U_i \setminus V_i) < m_i + 1$
for every $1 \le i \le t$. We conclude from  
the above computation of $\mu_Z(\sP_r)$ that the above limit converges
and its value is

\[
\frac{\# T}{\# H^0(Y, \sO_Y)} \stackrel{t}{\underset{i = 1}\prod}
\frac{1}{\zeta_{U_i \setminus V_i}(m_i +1) 
\stackrel{m_i -1}{\underset{e = 0}\prod}\zeta_{(V_i)_e}(m_i - e)}.
\]
Since this is clearly positive, we conclude the proof.
\end{proof}

We now state a Bertini theorem for singular schemes.
Let $T \subset \P^n_k$ be a finite set of points.
Assume that $U \subset \P^n_k$ is a subscheme which is a disjoint union of locally closed
subschemes $U_1, \ldots , U_t$ such that $U_i$ is smooth of pure dimension
$m_i \ge 0$ and $U_i \cap T = \emptyset$ for $1 \le i \le t$. 
Note that $U$ itself may not be smooth.
Let $Z \subset \P^n_k$ be a closed subscheme such that $T \cap Z =
\emptyset$ and $\tdim(Z \cap U_i) < m_i$ for $1 \le i \le t$. Define
\begin{equation}\label{eqn:B-sing-0}
\sP = \{f \in I^Z_{\homg}|T \cap H_f = \emptyset \ \mbox{and} \ 
f \notin \fm^2_{U,x} \ \mbox{for \ all \ closed \ points} \ x \in \, U\}.
\end{equation}

\begin{thm}\label{thm:PW-sing-gen}
Under the above assumptions, there exists $\sP' \subseteq \sP$ such that $\mu_Z(\sP') > 0$.
\end{thm}
\begin{proof}
We write $T = T_1 \amalg T_2$, where
$T_1$ consists of the closed points of $\P^n_k$ lying in $T$ and $T_2$ 
consists of the non-closed points of $\P^n_k$ lying in $T$.
We write $T_1 = \{P_1, \ldots , P_{r}\}$ and
let $T'_1 =  \stackrel{r}{\underset{i =1}\prod} k(P_i)^\times
\subset H^0(T_1, \sO_{T_1})$.
Letting
\[
\sP_0 = \{f \in I^Z_{\homg}|f|_{T_1}  \in T'_1 \ \mbox{and} \  U_i \cap H_f \ 
\mbox{is \ smooth \ of \ dimension} \ m_i -1 \ \forall \ 1 \le i \le t\},
\]
it follows from \propref{prop:Poonen} that $\mu_Z(\sP_0) > 0$.

For any point $x \in \P^n_k$, we let
$\sP_x = \{f \in I^Z_{\homg}| x \notin H_f\}$ and define
$\sP' = ({\underset{x \in T_2}\cap} \sP_x) \bigcap \sP_0$.
Since no point of $T_2$ is either closed in $\P^n_k$ or lies in $Z$,
it follows from \lemref{lem:positive-dim} that 
$\mu_Z(\sP_x) = 1$ for all $x \in T_2$. We conclude from
~\eqref{eqn:RS-main-1-0} and \lemref{lem:Elem} that
$\mu_Z(\sP') > 0$.
We only need to show that $\sP' \subseteq \sP$.
But this is an easy exercise and is left to the reader.
\end{proof}

\begin{cor}\label{cor:PW-sing}
  Let $X \subset \P^n_k$ be a subscheme. Let $Y \subset \P^n_k$ be a finite
  closed subscheme and $T \subset H^0(Y, \sO_Y)$ a nonempty subset. Let
 \begin{equation}\label{eqn:PW-sing-00}
   \sP = \{f \in S_{\homg}|f|_Y \in T \ \ \mbox{and} \ \
   f \notin \fm^2_{X,x} \ \mbox{for \ all \ closed \ points}
   \ x \in\, X\setminus Y\}.
\end{equation}
 Then there exists $\sP' \subseteq \sP$ such that
 $\mu(\sP') > 0$.
\end{cor}
\begin{proof}
  Use the fact that $X$ has a finite stratification by smooth subschemes.
  \end{proof}

  \subsection{Bertini theorem for normal crossing schemes}
  \label{sec:Extra}
In this subsection, we list some new consequences of \propref{prop:Poonen}
which could not be derived directly from the Bertini theorem of Wutz.
We mention these consequences because they are often very useful in
the theory of algebraic cycles and class field theory over finite fields
(e.g., see \cite{BK-1}, \cite{BKS} and \cite{GK-2}).
The results of this subsection are not used in this paper.
Readers interested only in the main results of this paper could therefore skip it.

We recall that a $k$-scheme $X$ of pure dimension $m$ with irreducible components $X_1, \ldots , X_r$ is said to be a strict normal crossing (snc) scheme if it is reduced and for all $\emptyset\neq J \subset \{1, \ldots , r\}$, the scheme theoretic intersection $X_J := {\underset{i \in J}\bigcap} X_i$ is either empty or regular of pure dimension $m + 1 - |J|$.  By a scheme of negative dimension, we shall always mean the empty scheme.

Let $Y \subset \P^n_k$ be a finite closed subscheme and $T \subset H^0(Y, \sO_Y)$ a nonempty subset. Let $V \subset \P^n_k$ be a smooth subscheme of pure dimension $m$ disjoint from $Y$. Let $U \subset V$ be an open subscheme and let $E \subset U$ be a strict normal crossing divisor (sncd). Let $E_J$ be defined as above and set $E_{\emptyset}= U$. Let $Z \subset \P^n_k$ be a closed subscheme such that $Y \cap Z = \emptyset$, $V \cap Z \subset U$ and $\tdim(Z \cap E_J) < m - |J|$ for all $J \subset \{1, \ldots , r\}$ (cf. \S~\ref{sec:B-sm*}). Define
\begin{equation}\label{eqn:PW-snc}
\sP = \{f \in I^Z_\homg|f|_Y \in T \ \mbox{and} \ V \cap H_f \ 
\mbox{is \ smooth \ and} \ E \cap H_f \ \mbox{is \ a \ sncd \ on} \ U \cap H_f\}.
\end{equation}

\begin{thm}\label{thm:snc-0}
Under the above assumptions, there exists $\sP' \subseteq \sP$ such that $\mu_Z(\sP') > 0$.
\end{thm}
\begin{proof}
Define $E'_J = E_J \setminus ({\underset{J_1 \supsetneq J}\cup} E_{J_1})$.
Then $U$ is a disjoint union of subschemes $E'_J$ such that
each $E'_J$ is smooth of dimension $m - |J|$ and $\tdim(Z \cap E'_J) < m -|J|$.
Let
\[
\sP'_1 = \{f \in I^Z_\homg|f|_Y \in T \ \mbox{and} \ E'_J \cap H_f \ 
\mbox{is \ empty \ or \ smooth \ of \ dimension} \ m - |J| -1 \ \forall \ 
J\}.
\]
Write $V \setminus U$ as a disjoint union of locally closed subschemes  
$F_1, \ldots , F_s$ such that each $F_j$ is smooth of dimension $n_j \ge 0$.
Let
\[
\sP'_2 = \{f \in I^Z_\homg|f|_Y \in T \ \mbox{and} \ F_j \cap H_f \ 
\mbox{is \ smooth \ of \ dimension} \ m_j-1 \ \forall \ 
1 \le j \le s\}
\]
and $\sP' = \sP'_1 \cap \sP'_2$.

It follows from \propref{prop:Poonen} that $\mu_Z(\sP') > 0$.
It suffices therefore to show that $\sP' \subseteq \sP$.
For this, we first note that $f \in \sP'_1$ implies that
$E \cap H_f$ is generically smooth. In particular, it is an
$R_0$-scheme. Since $E$ is a sncd in the smooth scheme $U$, it is Cohen-Macaulay.
It follows that $E \cap H_f$ is also Cohen-Macaulay.
We conclude that $E \cap H_f$ is reduced.
Other desired properties of $E \cap H_f$ for it to be sncd are immediate.

We next show the smoothness of $U \cap H_f$ for $f \in \sP'_1$.
This is clear away from $E$.
At a point in $E \cap H_f$, the smoothness of $U \cap H_f$ is an easy consequence of
a descending induction on $|J|$ and the elementary fact that
if a locally principal divisor $Y$ on a Noetherian scheme $X$ is regular at a point,
then $X$ is also regular at that point.

To conclude the proof of the theorem, it remains to show that
$V \cap H_f$ is smooth at points away from $U$ for $f \in \sP'_2$. 
It is enough to check this at the closed points of $V \setminus U$.
So let $x \in V \setminus U$ be a closed point. Then it must lie in
$F_j$ for a unique $j \in [1, s]$. Since $F_j$ and $F_j \cap H_f$ are both
smooth at $x$, it follows that the image of $f$ in the local ring $\sO_{F_j,x}$
does not lie in $\fm^2_{F_j,x}$. But then, it can not lie in $\fm^2_{V,x}$ either.
Since $V$ is smooth, this condition is equivalent to the assertion that
$V \cap H_f$ is smooth at $x$.
\end{proof}

\begin{cor}\label{cor:snc-0-dim}
Let $X \subset \P^n_k$ be a smooth 
scheme and let $U \subset X$ be an open subscheme.
Let $D \subset X$ be a closed subscheme such that $D \cap U$ is a
strict normal crossing divisor on $U$. 
Let $Z \subset \P^n_k$ be a closed subscheme such that
$Z \cap X$ is a finite reduced subscheme of $X$ contained in $U$ and
$Z \cap D = \emptyset$. Let
\[
\sP = \{f \in I^Z_\homg|X \cap H_f \ \mbox{is \ smooth \ and} \ D \cap U \cap H_f 
\ \mbox{is \ a \ sncd \ on} \ U \cap H_f\}.
\]
Then $\sP$ contains a subset $\sP'$ such that $\mu_Z(\sP') > 0$.
\end{cor}
\begin{proof}
Apply \thmref{thm:snc-0} with $Y = \emptyset$.
\end{proof}

\begin{remk}\label{remk:snc-int}
  If $X \subset \P^n_k$ is closed and integral of dimension $d \ge 2$
  in \corref{cor:snc-0-dim},
  then $X \cap H_f$ will be necessarily integral for every $f \in \sP$.
  This can be easily deduced from the Enriques-Severi-Zariski vanishing theorem
  (e.g., see \cite[Lemma~5.1]{Ghosh-Krishna})
  and the fact that $H^0(X, \sO_X)$ is a field.
  \end{remk}

\begin{remk}\label{remk:snc-int-0}
  The reader may note that there are concrete situations (in the theory of algebraic
  cycles, for instance) where \corref{cor:snc-0-dim}
  is applicable. We often need to apply Bertini theorems to a given quasi-projective
  scheme $X$ and a divisor $D$ on it. It may happen that even if $X$ is smooth,
  $D$ may not be sncd everywhere but only along a proper open subset of $X$.
  In this case, \corref{cor:snc-0-dim} ensures existence of hypersurface sections
  of $X$ having similar properties.
  \end{remk}

\section{Bertini-integrality over finite fields}\label{sec:Irr-2}
Our goal in this section is to prove the remaining Bertini theorems
over finite fields: Bertini for irreducibility and integrality.
This requires additional work.

In \cite{Poonen-3}, Charles and Poonen proved a
Bertini-irreducibility theorem over finite fields.
In this paper, we shall prove a generalization of their result where we add an
additional constraint on the hypersurface
sections that they contain a prescribed closed subscheme $Z$ (satisfying certain
necessary conditions) of the underlying ambient scheme.
The Bertini theorems of this kind are very useful in algebraic geometry.
For instance, in the study of algebraic cycles and class field theory,
one looks for `good' hypersurfaces which contain a given algebraic
cycle on the underlying scheme.
We shall combine the Bertini-irreducibility theorem with \corref{cor:B-reduced-fin}
to prove our main result: the Bertini-integrality theorem.
We remark that the latter result is completely new, even when the prescribed closed
subscheme is empty.

Our proof  of the Bertini-irreducibility theorem 
closely follows the `$Z = \emptyset$' case shown in \cite{Poonen-3}.
However, some new steps and additional arguments are required
to take care of the presence of the prescribed closed subscheme.

\subsection{Some Lemmas}\label{sec:Lem*}
For a Noetherian scheme $X$, recall that $\irr (X)$ denotes the set of 
irreducible components of $X$. Let $X^{(i)} = \{x \in X| \dim(\sO_{X,x}) = i\}$.
We fix a finite field $k$ and an algebraic closure $\ov{k}$ of $k$.

\begin{lem}\label{lem:open}
Let $X$ be a subscheme of $\P^n_{\ov{k}}$. Let $U \subset X$ be a dense open subscheme.
Let $Z \subset \P^n_k$ be a closed subscheme such that
$Z_{\ov{k}} \cap X$ has codimension at least two in every irreducible component of
$X$. Then, for $f$ in a subset of $I^Z_{\homg}$ of density one, there is a bijection
$\irr (X_f) \to \irr (U_f)$ sending $D$ to $U_f \cap D$.
\end{lem}
\begin{proof}
  Let $T = \{x \in X \setminus U| \dim(\sO_{X,x}) = 1\}$.
  Since $U \subset X$ is dense, $T$ must be a finite set.
  It is also clear that  $T \cap Z_{\ov{k}} = \emptyset$.
Therefore, for any $t \in T$, we get that $\ov{\{t\}} \not\subset Z_{\ov{k}}$.
It follows from \lemref{lem:positive-dim} that
there is a subset $\sP_t \subset I^Z_{\homg}$ of density one, none of
whose elements vanishes on $\ov{\{t\}}$.
We let $\sP = {\underset{t \in T}\bigcap} \sP_t$ so that
$\mu_Z(\sP) = 1$ by \lemref{lem:Elem}. It is easy to check that
every $f \in \sP $ satisfies the desired property (e.g.,
see the proof of \cite[Lemma~3.3]{Poonen-3}).
\end{proof}

The following result generalizes a weaker version of
\cite[Lemma~3.5]{Poonen-3} to hypersurfaces containing 
a prescribed closed subscheme. But we will show that this weaker
version is sufficient for the proof of the Bertini-irreducibility theorem.

\begin{lem}\label{lem:Sing-finite}
Let $Y$ be a smooth irreducible subscheme of $\P^n_k$ of
dimension $m \ge 1$. Let $Z \subset \P^n_k$ be a closed subscheme such that
$Y \cap Z = \emptyset$. Let $X \in \irr (Y_{\ov{k}})$. Then, for
$f$ in a subset of $I^Z_{\homg}$ of density one, the scheme
$(X_f)_{\rm sing}$ is finite.
\end{lem}
\begin{proof}
Let $\sP = \{f \in I^Z_{\homg}| (Y_f)_{\rm sing} \ \mbox{is \ finite}\}$.
For $f \notin \sP$, we see that $(Y_f)_{\rm sing}$ has positive
dimension, and hence it contains closed points of arbitrarily high degrees.
It follows that $f$ is contained in the set
\[
\sQ^{\rm high} := {\underset{d \ge 1}\bigcup}
\{f \in I_{Z,d}|Y_f \ \mbox{contains \ a \ closed \ point} \ y \ \mbox{of} 
\ \mbox{degree} \ > \frac{d - c_Z}{m+1}
\]
\[
\hspace*{4cm} \mbox{such \ that} \
Y_f \ \mbox{is \ not \ smooth \ of \ dimension} \ m-1  \ \mbox{at} \ y\},
\]
where $c_Z$ is as in ~\eqref{eqn:surjective}. We conclude that
$\sP \cup \sQ^{\rm high} = I^Z_{\homg}$. On the other hand, \cite[Lemma~4.2]{Poonen-2}
says that $\ov{\mu}_Z(\sQ^{\rm high}) = 0$. In particular,
$\mu_Z(\sQ^{\rm high}) = 0$. We must therefore have
$\mu_Z(\sP) = 1$ (see the proof of \lemref{lem:Elem}).
It remains to show that $(X_f)_{\rm sing}$ is finite for every
$f \in \sP$. We fix an $f \in \sP$ and let $W = (Y_f)_{\rm sing}$. Then 
$W$ is finite and $Y_f \setminus W$ is smooth open in $Y_f$. 
In particular, $W_{\ov{k}}$ is finite and
$(Y_f)_{\ov{k}} \setminus W_{\ov{k}} =
(Y_f \setminus W)_{\ov{k}}$ is smooth.
If $X \in \irr (Y_{\ov{k}})$, then $X_f \setminus W_{\ov{k}}$ is
an open subset of $(Y_f)_{\ov{k}} \setminus W_{\ov{k}}$ and 
must therefore be smooth. 
\end{proof}

\subsection{The dimension two case}\label{sec:Dim-2}
The following is a version of Bertini-irreducibility 
theorem in dimension two. Let $k$ be a finite field and $\ov{k}$ an algebraic closure of $k$.

\begin{lem}\label{lem:Dim-2-B}
Let $X \subset \P^n_k$ be a closed integral subscheme of dimension two
and let $Z \subset \P^n_k$ be a closed subscheme such that $Z \cap X$
is finite. Then, for $f$ in a subset of $I^Z_{\homg}$ of density one, 
there is a bijection $\irr (X_{\ov{k}}) \xrightarrow{\simeq} 
\irr ((X_f)_{\ov{k}})$ which sends $D$ to $D \cap (X_f)_{\ov{k}}$.
\end{lem}
\begin{proof}
We shall prove this lemma using \cite[Proposition~4.1]{Poonen-3}.
We consider the commutative diagram
\begin{equation}\label{eqn:Dim-2-B-0}
\xymatrix@C.8pc{
H^0(\P^n_k, \sI_Z(d)) \ar[r]^-{\alpha_d} \ar[d]_-{\beta_d} &
H^0(X, \sI_Z \cdot \sO_X(d)) \ar[d]^-{\gamma_d} \\
H^0(\P^n_k, \sO_{\P^n_k}(d)) \ar[r]^-{\delta_d} &
H^0(X, \sO_X(d))} 
\end{equation}
for $d \ge 0$ obtained by the obvious restrictions and inclusions of sheaves.
The vertical arrows are injective for all $d$ and there exists $d_0 \gg 0$
such that the horizontal arrows are surjective for all $d \ge d_0$.

Define the subsets $S_{X, \homg} = 
{\underset{d \ge 1}\bigcup} H^0(X, \sO_X(d))$ and $I^Z_{X, \homg} = 
{\underset{d \ge 1}\bigcup} H^0(X, \sI_Z \cdot \sO_X(d))$ of
$\wt{S}_X = {\underset{d \ge 0}\bigoplus} H^0(X, \sO_X(d))$.
The diagram~\eqref{eqn:Dim-2-B-0} gives rise to a commutative diagram

\begin{equation}\label{eqn:Dim-2-B-1}
\xymatrix@C.8pc{
I^Z_{\homg} \ar[r]^-{\alpha} \ar[d]_-{\beta} & I^Z_{X, \homg} \ar[d]^-{\gamma} \\
S_{\homg} \ar[r]^-{\delta} & S_{X, \homg}}
\end{equation}
of sets in which the vertical arrows are injective and
the horizontal arrows are surjective in degrees $d \ge d_0$.

We define two new density functions $\mu'$ and $\mu'_Z$ on the
subsets of $S_{X, \homg}$ and $I^Z_{X, \homg}$ as follows.
Given $\sP_1 \subset S_{X, \homg}$ and $\sP_2 \subset I^Z_{X, \homg}$,
we let
\begin{equation}\label{eqn:Dim-2-B-2}
\mu'(\sP_1) = {\underset{d \to \infty}{\rm lim}} \frac{\#(\sP_1 \cap 
H^0(X, \sO_X(d)))}{\# H^0(X, \sO_X(d))} \ 
\mbox{and} \ 
\mu'_Z(\sP_2) = {\underset{d \to \infty}{\rm lim}} \frac{\#(\sP_2 \cap 
H^0(X, \sI_Z \cdot \sO_X(d)))}{\# H^0(X, \sI_Z \cdot \sO_X(d))},
\end{equation}
if the limits exist.

We consider the sets
$\sP = \{f \in I^Z_{\homg}| \irr (X_{\ov{k}}) \to \irr ((X_f)_{\ov{k}}) \
\mbox{is \ a \ bijection}\}$ and
$\sP' = \{f \in S_{\homg}| \irr (X_{\ov{k}}) \to \irr ((X_f)_{\ov{k}}) \
\mbox{is \ a \ bijection}\}$.
Then $\sP = \beta^{-1}(\sP')$.
Let $\sP''' = \delta(\sP')$ and $\sP'' = \gamma^{-1}(\sP''')$.
If there are $f, g \in S_{\homg}$ such that $g \in \sP'$ and
$\delta(f) = \delta(g)$, then both $f$ and $g$ must have the same degree
(say, $d$) unless $\delta(f) = \delta(g) = 0$. 
In the latter case, the equality $X_f = X_g$ is automatic.
In the former case, the exact sequence
\[
0 \to H^0(\P^n_k, \sI_X(d)) \to H^0(\P^n_k, \sO_{\P^n_k}(d)) 
\xrightarrow{\delta_d} H^0(X, \sO_X(d))
\]
implies that $f-g \in H^0(\P^n_k, \sI_X(d))$
so that $X_f = X_g$. This forces $f$ to also lie in $\sP'$.
It follows therefore that $\delta^{-1}(\sP''') = \sP'$.
In particular, $\alpha^{-1}(\sP'') = \sP$.

\cite[Proposition~4.1]{Poonen-3} says that $\mu(\sP') = 1$.
Since $\delta_d$ is surjective for $d\ge d_0$, we get
\[
\begin{array}{lll}
\mu'(\sP''') & = &  
{\underset{d \to \infty}{\rm lim}} \frac{\#(\sP''' \cap 
H^0(X, \sO_X(d)))}{\# H^0(X, \sO_X(d))} \\
& = &  {\underset{d \to \infty}{\rm lim}} \frac{\#
\delta^{-1}_d(\sP''' \cap 
H^0(X, \sO_X(d)))}{\# \delta^{-1}_d(H^0(X, \sO_X(d)))} \\
& = & {\underset{d \to \infty}{\rm lim}} \frac{\#(\sP' \cap S_d)}{\# S_d} \\
& = & \mu(\sP') = 1.
\end{array}
\]
The short exact sequence of sheaves
\[
0 \to \sI_Z \cdot \sO_X(d) \to \sO_X(d) \to \sO_{Z\cap X}(d) \to 0
\]
and the finiteness of $Z \cap X$ together imply that
there exists an integer $b \ge 1$ such that
$\# {\rm Coker}(\gamma_d) \le q^b$ for all $d \ge 1$.

Let $\epsilon > 0$ be given. Since $\mu'(\sP''') = 1$, 
there exists $d_1 \gg 1$ such that for all $d \ge d_1$, we have
$\frac{\#(\sP''' \cap H^0(X, \sO_X(d)))}{\# H^0(X, \sO_X(d))} > 
1 - \frac{\epsilon}{q^b}$.
Equivalently, 
$\frac{\#((\sP''')^c \cap H^0(X, \sO_X(d)))}{\# H^0(X, \sO_X(d))}
< \frac{\epsilon}{q^b}$, where $(\sP''')^c$ is the complement of
$\sP'''$ in $S_{X, \homg}$.

Since $(\sP'')^c \cap H^0(X, \sI_Z \cdot \sO_X(d)) =
\gamma^{-1}_d((\sP''')^c \cap H^0(X, \sO_X(d)))$ and $\gamma_d$ is
injective, we get
$\frac{\#((\sP'')^c \cap H^0(X, \sI_Z \cdot \sO_X(d)))}{\# H^0(X, \sO_X(d))}
< \frac{\epsilon}{q^b}$.
Since $\# H^0(X, \sO_X(d)) \le q^b (\# H^0(X, \sI_Z \cdot \sO_X(d)))$,
it follows that
$\frac{\#((\sP'')^c \cap H^0(X, \sI_Z \cdot \sO_X(d)))}
{\# H^0(X, \sI_Z \cdot \sO_X(d))} < \epsilon$. Equivalently, we get
\begin{equation}\label{eqn:Dim-2-B-3}
\frac{\#(\sP''\cap H^0(X, \sI_Z \cdot \sO_X(d)))}
{\# H^0(X, \sI_Z \cdot \sO_X(d))} > 1 - \epsilon \ \mbox{for \ all} \
d \ge d_1.
\end{equation}
This shows that $\mu'_Z(\sP'') = 1$.

To conclude, we note that $\alpha_d$ is surjective for all $d \ge d_0$
and we have shown that $\sP = \alpha^{-1}(\sP'')$.
This implies that
\[
\begin{array}{lll}
\mu_Z(\sP) & = & {\underset{d \to \infty}{\rm lim}}
\frac{\#(\sP \cap I_{Z,d})}{\# I_{Z,d}} = 
{\underset{d \to \infty}{\rm lim}} 
\frac{\# \alpha^{-1}_d(\sP'' \cap H^0(X, \sI_Z \cdot \sO_X(d)))}
{\# \alpha^{-1}_d(H^0(X, \sI_Z \cdot \sO_X(d)))} \\
& = & {\underset{d \to \infty}{\rm lim}} 
\frac{\# (\sP''\cap H^0(X, \sI_Z \cdot \sO_X(d)))}
{\# H^0(X, \sI_Z \cdot \sO_X(d))} =  \mu'_Z(\sP'') = 1.
\end{array}
\]
\end{proof}

\subsection{The general case}\label{sec:Final}
We need some lemmas to deduce the general case of the Bertini-irreducibility
theorem over finite fields from the case of surfaces. 
We let $k$ be a finite field and $\ov{k}$ an algebraic closure of $k$.

The following lemma
is a direct generalization of a weaker version of \cite[Lemma~5.3]{Poonen-3}
to hypersurfaces containing a prescribed closed subscheme.

\begin{lem}\label{lem:Reduction-1}
Let $Y$ be a smooth irreducible subscheme of $\P^n_k$
of pure dimension $m \ge 3$ and let $X \in \irr (Y_{\ov{k}})$. Let $Z \subset \P^n_k$
be a closed subscheme such that $Y \cap Z = \emptyset$ and
$\ov{Y} \cap Z$ has codimension at least two in $\ov{Y}$.
Then there exists a hypersurface $J \subset \P^n_k$ 
satisfying the following.

\begin{enumerate}
\item
$X \cap J_{\ov{k}}$
is irreducible of dimension $m-1$, $\dim(J_{\ov{k}} \cap (\ov{X} \setminus X))
\le m-2$ and $\dim(\ov{X} \cap  J_{\ov{k}} \cap Z_{\ov{k}}) \le m-3$.
\item
The subset 
$\sP = \{f \in I^Z_{\homg}| X_f \ \mbox{is \ irreducible \ or \ } \
X_f \cap J_{\ov{k}} \ \mbox{is \ reducible}\}$
of $I^Z_{\homg}$ has density one.
\end{enumerate}
\end{lem}
\begin{proof}
Let $\pi \colon \P^n_{\ov{k}} \to \P^n_k$ be the projection map.
We know that the Galois group ${\rm Gal}({\ov{k}}/{k})$ acts on $\P^n_{\ov{k}}$
and on the set of all its subsets. We can write
$\irr (Y_{\ov{k}}) = \{\sigma_1(X), \ldots , \sigma_r(X)\}$, where
$\sigma_i \in {\rm Gal}({\ov{k}}/{k})$ and $\sigma_1 = {\rm id}$.
Since $Y$ is smooth, all elements of $\irr (Y_{\ov{k}})$ are mutually
disjoint. It is also easy to see that $\pi^{-1}(\ov{Y}) =
\ov{(Y_{\ov{k}})}$. Indeed, if there is an open subset $U \subset \P^n_{\ov{k}}$ 
which meets $\pi^{-1}(\ov{Y})$ and does not meet $Y_{\ov{k}}$, then
$\pi(U)$ is an open subset of $\P^n_k$ which meets $\ov{Y}$ but not
$Y$. But this is not possible.
We therefore get 
$\irr ((\ov{Y})_{\ov{k}}) = 
\{\sigma_1(\ov{X}), \ldots , \sigma_r(\ov{X})\}$,
$(\ov{Y})_{\ov{k}} = \stackrel{r}{\underset{i =1}\cup}
\sigma_i(\ov{X})$ and $\pi^{-1}(\ov{Y} \cap Z) =
\stackrel{r}{\underset{i =1}\cup} (\sigma_i(\ov{X}) \cap Z_{\ov{k}})$.
This implies that $\dim(\ov{X}) = m$ and
$\dim(\ov{X} \cap Z_{\ov{k}}) = \dim(\ov{Y} \cap Z) \le m-2$.
The rest of the proof is identical to that of \cite[Lemma~5.3]{Poonen-3}, using \lemref{lem:positive-dim} instead of \cite[Lemma~3.1]{Poonen-3} and \lemref{lem:Sing-finite} instead of \cite[Lemma~3.5]{Poonen-3}.
\end{proof}

\begin{lem}\label{lem:base-change}
Let $Y \in \Sch_k$ be an integral scheme and let
$Y' \in \irr (Y_{\ov{k}})$. Then $Y'$ maps onto $Y$ under the
projection map $\pi \colon Y_{\ov{k}} \to Y$.
\end{lem}
\begin{proof}
  This follows easily from the proof of \lemref{lem:B-irr-inf}.
\end{proof}

\begin{lem}\label{lem:Reduction-2}
Let $X \subset \P^n_k$ be an irreducible subscheme of 
dimension $m \ge 2$. Let $Z \subset \P^n_k$ be a closed subscheme such that
$Z \cap \ov{X}$ has codimension at least two in $\ov{X}$. 
Then there is a subset $\sP \subset I^Z_{\homg}$ such that
$\mu_Z(\sP) = 1$ and every $f \in \sP$ defines 
a bijection $\irr (X_{\ov{k}}) \xrightarrow{\simeq} \irr ((X_f)_{\ov{k}})$ 
sending $D$ to $D \cap (X_f)_{\ov{k}}$. 
\end{lem}
\begin{proof}
We can assume $X$ to be reduced, and hence integral, in order to prove
the lemma. We shall prove the lemma by induction on $m$.
The base case follows easily from Lemmas~\ref{lem:open} and
~\ref{lem:Dim-2-B}. We can
therefore assume that $m \ge 3$.

Since $X_{\rm sm}$ is dense open in $X$, there is a bijection
$\irr (X_{\ov{k}}) \xrightarrow{\simeq} \irr ((X_{\rm sm})_{\ov{k}})$. It follows by
\lemref{lem:open} that for $f$ in a density one subset of
$I^Z_{\homg}$, there is a bijection $\irr ((X_f)_{\ov{k}}) \xrightarrow{\simeq}
\irr ((X_{\rm sm})_f)_{\ov{k}}$. We can therefore assume that $X$ is
integral and smooth.
Lastly, if $X' = X \setminus Z$, then $X_{\ov{k}} \setminus X'_{\ov{k}}$
has codimension at least two in $X_{\ov{k}}$. In particular,
$X_{\ov{k}} \setminus X'_{\ov{k}}$ does not contain any  element of
$\irr (X_{\ov{k}})$ and $((X \setminus X') \cap H_f)_{\ov{k}}$ 
does not contain any element of $\irr((X_f)_{\ov{k}})$. 
It suffices therefore to prove the lemma for $X'$.
We can therefore assume without loss of generality that $X$ is 
an integral and smooth subscheme of $\P^n_k$ such that $X \cap Z =
\emptyset$.
The rest of the proof is identical to that of \cite[Proposition~5.4]{Poonen-3},
using Lemmas~\ref{lem:Reduction-1} and ~\ref{lem:base-change} instead
of Lemma~5.3 of op. cit..
\end{proof}

We can now prove the main results of \S~\ref{sec:Irr-2}.

\begin{thm}\label{thm:B-irr-fin}
Let $k$ be a finite field and $X \subset \P^n_k$ a subscheme of dimension $m \ge 2$.
Let $Z \subset \P^n_k$ be a closed subscheme such that
$Z \cap \ov{X}$ has codimension at least two in $\ov{X}$.
Assume that $X$ is irreducible (resp. geometrically irreducible).
Let
\[
\sP = \{f \in I^{{Z}}_{\homg}| H_f \cap X \ \mbox{is \ irreducible} 
\ (\mbox{resp. \ geometrically \ irreducible})\}.
\]
Then $\mu_{{Z}}(\sP)  = 1$.
\end{thm}
\begin{proof}
If $X$ is geometrically irreducible, then the theorem is an immediate consequence of
  \lemref{lem:Reduction-2}. We therefore have to consider the case when
  $X$ is irreducible but not necessarily geometrically irreducible.
   We fix an algebraic closure $\ov{k}$ of $k$.
We showed in \lemref{lem:Reduction-2} that there is a subset $\sP' \subset I^Z_{\homg}$
such that $\mu_Z(\sP') = 1$ and for every $f \in \sP'$, we have a bijection
$\irr (X_{\ov{k}}) \xrightarrow{\simeq} \irr ((X_f)_{\ov{k}})$.
It follows easily from \lemref{lem:base-change} (see the proof of \lemref{lem:B-irr-inf}) that $X_f$ must be irreducible
if $f \in \sP'$. This implies that $\mu_{{Z}}(\sP)  = 1$.
\end{proof}  

Let $k$ be a finite field and $X$ an integral subscheme of $\P^n_k$ of 
dimension $m \ge 2$.
Let $Z \subset \P^n_k$ be a closed subscheme 
such that $Z \cap \ov{X}$ has codimension at least two in $\ov{X}$.
Assume that $Z$ does not contain any generic point of $X_\sing$ and 
$Z \cap \Sigma^1_X = \emptyset$ (cf. ~\eqref{eqn:SR-locus}). Assume further that
$\tdim(Z\cap X_{\sm}) < \dim(X)$. Let
$T \subset \P^n_k$ be a finite set such that $T \cap Z = \emptyset$.
Let $\sP_{\rm int} \subset I^{{Z}}_{\homg}$ be the subset such that $f \in \sP_{\rm int}$
if and only if $T \cap H_f = \emptyset$, $X \cap H_f$ is integral and
$X_\sm \cap H_f$ is smooth.
We define $\sP_{\rm gint} \subset I^{{Z}}_{\homg}$ by replacing the
integrality condition in the definition of $\sP_{\rm int}$ by
geometric  integrality.

\begin{thm}\label{thm:B-int-fin}
Under the above assumptions, there exists $\sP' \subseteq \sP_{\rm int}$ such that $\mu_Z(\sP') > 0$.
If $X$ is geometrically integral, then there exists $\sP'' \subseteq \sP_{\rm gint}$ such that $\mu_Z(\sP'') > 0$.
\end{thm}
\begin{proof}
Combine \corref{cor:B-reduced-fin} (see its proof), \thmref{thm:B-irr-fin} and
  \lemref{lem:Elem}.
\end{proof}

\section{Bertini theorems over a dvr}\label{sec:dvr*}
We set up the notations which will be used throughout this section.
Let $A$ be a discrete valuation ring with maximal ideal 
$\fm = (\pi)$. Let $K$ denote the quotient field and $k$ the
residue field of $A$. Let $S = \Spec(A)$.
We let $S' = A[x_0, \ldots , x_n]$ so that $\P^n_A = \Proj_A(S') = \P_A(V)$,
where $V = Ax_0 \oplus \cdots \oplus Ax_n$ is a free $A$-module of rank $n+1$.
We let $S'_\eta = S' \otimes_A K$ and $S'_s = S' \otimes_A k$.
We define a hypersurface $H \subset \P^n_A$ of degree $d$ to be
a closed subscheme of the form $\Proj_A({S'}/{(f)})$, where
$f \in S'_d$ is a homogeneous polynomial of degree $d$ not all of whose
coefficients are in $\fm$. 
For $f \in S'$, we let $\ov{f}$ denote its image under the surjection
$S' \surj S'_s$. For any subscheme $Z \subset \P^n_A$, we  let
$Z_\eta$ (resp. $Z_s$) denote the generic (resp. special) fiber of $Z$ over $S$. 
If $Z \subset \P^n_A$ is closed, we let $\sI_Z$ denote the sheaf of ideals on
$\P^n_A$ defining $Z$. We define $\sI_{Z_\eta}$ and $\sI_{Z_s}$ similarly.
We let $I_s \subset S'_s$ be the homogeneous ideal
defining $Z_s$. 

\subsection{Specialization of hypersurfaces}\label{sec:Esp}
The goal of this subsection is to prove some technical results
which will allow us to
reduce Bertini theorems over $A$ to such results over the quotient and residue fields
of $A$.

For an integer $N \ge 1$, let ${\rm sp} \colon \P^N_K(K) \to \P^N_k(k)$
be the standard specialization map. This takes a $K$-rational point
$x$ to the restriction of the closure $\ov{\{x\}}$ in $\P^N_S$ to the
special fiber $\P^N_k$. Note that this map is well defined because
$\P^N_S$ is projective over $S$. In precise terms, this map is defined
as follows. Let $x = [a_0, \ldots , a_N] \in \P^N_K(K)$. We let
$l = {\rm min}_{0 \le i \le N} \ v(a_i)$, where $v \colon K \to \Z$ is the
normalized discrete valuation with valuation ring $A$. Then
$a'_i := \pi^{-l}a_i \in A$ and not all of them lie in $\fm$. It is clear that
${\rm sp}(x) = [\ov{a'_0},  \ldots , \ov{a'_N}] \in \P^N_k(k)$.
The following is elementary.

\begin{lem}\label{lem:Rational}
  If $x \in \P^N_K$ is a closed point such that $\ov{\{x\}} \cap \P^N_k = \{x'\}$
  with $x' \in \P^N_k(k)$, then $x \in \P^N_K(K)$. The map ${\rm sp}$ 
is surjective.
\end{lem}
\begin{proof}
  The second part can be checked directly. Since the projection map
  $\ov{\{x\}} \to S$ is finite and dominant, the first part follows
  from the general commutative algebra statement that
if $f \colon A \to A'$ is a finite and injective homomorphism 
between Noetherian rings (with $A$ as above) such that the
induced map $k \to A' \otimes_A k$ is an isomorphism, then $f$ is an
isomorphism. This statement, in turn, is easily deduced from Nakayama's lemma (e.g.,
see \cite[Theorem~2.2]{Matsumura}).
\end{proof}

\begin{lem}\label{lem:Non-empty-sp}
Given any point $x \in \P^N_k(k)$ and nonempty open subset
$U \subset \P^N_K$, the intersection ${\rm sp}^{-1}(x) \cap U(K)$
is infinite.
\end{lem}
\begin{proof}
Let $x = [\ov{a_0}, \ldots , \ov{a_N}]$, where $a_i \in A$ for every $i$
and $a_i \in A^\times$ for some $i$. We assume that $a_0 \in A^\times$
as the arguments for all cases are identical.
We let $W:= {\rm sp}^{-1}(x) \cap U(K)$. As $U \subset \P^N_K$ is dense open, 
${\rm sp}^{-1}(x) \setminus W$ is contained in a hypersurface
$H_f \subset \P^N_K$ for some nonzero homogeneous polynomial
$f \in K[x_0, \ldots , x_N]$. We let $g(x_1, \ldots , x_N) =
f(a_0, a_1 + \pi x_1, \ldots , a_n + \pi x_N) \in K[x_1, \ldots , x_N]$.
It is clear that $g$ is a nonzero polynomial.

Now, there is an inclusion $A^N \subset {\rm sp}^{-1}(x)$, where
an element $(c_1, \ldots , c_N) \in A^N$ is identified with
the point $[a_0, a_1 + \pi c_1, \ldots , a_N + \pi c_N] \in {\rm sp}^{-1}(x)$.
Under this inclusion, we see that $g$ vanishes on $A^N \setminus W$.
But this forces $W$ to be infinite by \lemref{lem:vanishing-locus}
(where we take $I_1 = \cdots = I_N = A \subset K$).
\end{proof}

\begin{lem}\label{lem:vanishing-locus}
Let $h \in K[x_1, \ldots , x_N]$ be a nonzero polynomial and let
$I_1, \ldots , I_N$ be infinite subsets of $K$.
Let $W \subset I_1 \times \cdots \times I_N$ be a finite set.
Then $h$ can not vanish everywhere on $(I_1 \times \cdots \times I_N)
\setminus W$. 
\end{lem}
\begin{proof}
  If $h$ vanishes on $(I_1 \times \cdots \times I_N) \setminus W$, then
  it will vanish  everywhere on 
  $I'_1 \times I_2 \times \cdots \times I_N$, where $I'_1$ is the complement of
  the projection of $W$ on $I_1$. We can thus reduce the lemma to the case when
  $W = \emptyset$.
  This latter case is an easy exercise using induction on $N$. 
\end{proof}

We now let $Z \subset \P^n_A$ be a closed subscheme defined by a
homogeneous ideal $I \subset S'$ such that $Z$ is flat over $S$
(e.g., $Z$ is reduced and none of its irreducible components lie in $\P^n_k$).

\begin{lem}\label{lem:Non-empty-sp-Z} 
For all $d>>0$, and for any nonzero homogeneous polynomial $f \in I_{s}$ of degree $d$, the set ${\rm sp}^{-1}(H_f) \cap U(K)$ is infinite for any nonempty open subset $U \subset \P_K(H^0(\P^n_K, \sI_{Z_\eta}(d)))$.
\end{lem}
\begin{proof}
  Since $Z$ is flat over $S$, the canonical map
  $\sI_Z \otimes_A k \to \sI_{Z_s}$ of coherent sheaves on $\P^n_A$ is an isomorphism.
  In particular, the canonical homomorphism
  $H^0(\P^n_A, \sI_Z(d))\otimes_A k \to H^0(\P^n_k, \sI_{Z_s}(d))$ is an isomorphism for
  all $d \gg 0$. Equivalently,  under the structure map
$\phi_{Z,d} \colon \P_A(H^0(\P^n_A, \sI_Z(d))) \to S$, the special fiber coincides
with $\P_k(H^0(\P^n_k, \sI_{Z_s}(d)))$ for all $d \gg 0$.
Applying \lemref{lem:Non-empty-sp} to $\phi_{Z,d}$, we conclude the proof.
\end{proof}

In the rest of this section, we shall combine the above results with
the Bertini theorems over fields to prove analogous theorems
over $A$.

\subsection{Bertini-regularity over $S$}\label{sec:B-reg**}
We shall prove our Bertini theorems over $A$ under the following assumptions.

\begin{notat}\label{notat:BT-13}
  Let $\sX \inj \P^n_S$ be a equidimensional connected quasi-projective scheme over $S$ whose
  every irreducible component has dimension at least two.
Let $\phi \colon \sX \to S$ be the structure map.
Assume that $\phi$ is surjective.
Let $\sX_{\eta} = \sX \times_S \{\eta\}, \ \sX_s = \sX \times_S \{s\}$ and
$X = (\sX_s)_\red$. Let $\ov{\sX}$ denote the scheme-theoretic
closure of $\sX$ in $ \P^n_S$.
\end{notat}

The first main result of this section is the following.
This was shown earlier in \cite[Theorem~1]{JS} when $X$ is a sncd on $\sX$, $\sX$ is regular and flat over $A$, and
$k$ is either infinite or (a weaker version of) $A$ is strict Henselian,
in \cite[Theorem~4.2]{SS} when $X$ is a sncd on $\sX$, its irreducible
components are smooth over $k$, and $\sX$ is regular, projective and flat over $A$, and in \cite[Proposition~2.3]{BK} when
$k$ is infinite and perfect and $\sX$ is regular, projective and flat over $A$.

\begin{thm}\label{thm:B-base-reg}
  In the situation of \notbref{notat:BT-13}, there exists an integer $d_0 \gg 0$ such that for
  all $d \ge d_0$, we can find infinitely many hypersurfaces $H \subset \P^n_S$ of degree
$d$ for which $\sX_\reg \cap H$ is regular. If the generic fiber
of $\sX_\reg$ is smooth, then so is the generic fiber of $\sX_\reg \cap H$.
\end{thm}
\begin{proof}
We can replace $\sX$ by $\ov{\sX}$ to prove the theorem. We therefore
assume that $\sX$ is projective over $S$. We can also assume that
$\sX_\reg \neq \emptyset$ because there is nothing to prove otherwise.
We shall now prove the theorem in the following steps.

\vskip .2cm

{\bf{Step~1:}}
Suppose first that $\sX_\reg \subset \sX_\eta$ so that $\sX_\reg = (\sX_\eta)_\reg$.
We can then apply \lemref{lem:B-regularity} to get a dense
open subscheme $\sU \subset \P_K(S'_{\eta,d})$ for every $d \ge 1$
such that all $H \in \sU(K)$ have the property that 
$H \cap \sX_\reg$ is regular. We let $\sU' \subset \P_A(S'_d)$ be
the complement of the Zariski closure of $\P_K(S'_d) \setminus \sU$ in $\P_A(S'_d)$.
It is then clear that
$H \cap \sX_\reg$ is regular for every $H \in \sU'(S)$.
If $\sX_\reg$ is smooth, then $\sX_\reg \cap H$ is smooth
by \cite[Theorem~1]{KA}.

\vskip .2cm

{\bf{Step~2:}}
Suppose now that $\phi \colon \sX_\reg \to S$ is surjective.
Since $(\sX_\eta)_\reg = \sX_\reg \cap \sX_\eta$, \lemref{lem:B-regularity} again
says that there is a dense open subscheme $\sU \subset \P_K(S'_{\eta,d})$ 
for every $d \ge 1$ such that every $H_\eta \in \sU(K)$ has the property 
that $\sX_\reg \cap H_\eta$ is regular. If $\sX_\reg \cap \sX_\eta$ 
is smooth, then $\sX_\reg \cap H_\eta$ is moreover smooth by 
\cite[Theorem~1]{KA}.

For $d \ge 1$, we let $F_d \subset |H^0(\P^n_k, \sO_{\P^n_k}(d))|(k)$ be the subset
consisting of
hypersurfaces $H \subset \P^n_k$ having the property that if $f \in S'_{s,d}$ is the
defining homogeneous polynomial of $H$, then the image of $f$ in $\sO_{X,x}$ is not in
$\fm^2_{X, x}$ for any closed point $x \in X$.

\vskip .2cm

{\bf{Claim:}} $|F_d| > 0$ for all $d \gg 0$. 

\vskip .2cm

The claim is a direct consequence of \corref{cor:PW-sing} 
when $k$ is finite. So we assume that $k$ is infinite.
We can write $X$ as a 
disjoint union of irreducible subschemes 
$X = \stackrel{r}{\underset{i = 1}\amalg} U_i$ of $\P^n_k$
such that each $U_i$ is regular. By \lemref{lem:B-regularity}, we can find a 
dense open subscheme $\sU' \subset |H^0(\P^n_k, \sO_{\P^n_k}(d))|$
for every  $d \ge 1$ such that for all $H \in \sU'(k)$ and $1 \le i \le r$,
one has that $H \cap U_i$ is regular and has codimension one in $U_i$.

Let $H \in \sU'(k)$ and let it be defined by the homogeneous
polynomial $f \in S'_{s,d}$.
If $x \in U_i$ is a closed point of $X$, then it is clear that
the image of $f$ in $\sO_{U_i,x}$ can not lie in
$\fm^2_{U_i,x}$. This implies that the image of $f$ in $\sO_{X,x}$ can not lie
in $\fm^2_{X, x}$. Note that $\sU'(k)$ is infinite because $k$ is infinite and
$\sU'$ is a rational $k$-variety. Since $\sU'(k) \subset F_d$, the claim follows.

\vskip .2cm

Let $\sU \subset \P_K(S'_{\eta,d})$ be the open subscheme chosen in the
beginning of Step~2 and set 
$F'_d = {\rm sp}^{-1}(F_d) \cap \sU(K) \subset \P_K(S'_{\eta,d})(K)$.
\lemref{lem:Non-empty-sp} says that $F'_d$ is infinite.
Given any $H_\eta \in F'_d$, we let $H$ be its Zariski closure in 
$\P_A(S'_{d})$. Then $H \in \P_A(S'_d)(S)$ so that it is defined
by a homogeneous polynomial $f \in S'_d$. We let $\sX' = \sX \cap H$.

\vskip .2cm

{\bf{Step~3:}}
We now show that $\sX'$ satisfies the
properties asserted in the theorem. Since we already showed this in Step~2
for $\sX'_\eta$, we only need to show that
$\sX'_{\sing} \cap \sX_\reg \cap X = \emptyset$.

By \cite[Exercise~8.2.17, Corollary~8.2.38]{QL}, $\sX'_\sing \cap \sX_\reg \cap X$
is a closed subset of $\sX' \cap \sX_\reg \cap X = \sX_\reg \cap H \cap X$.
Since the latter is an open subset of the projective scheme $H \cap X$ over $k$,
one deduces that if $\sX'_\sing \cap \sX_\reg \cap X \neq \emptyset$, then it
must contain a point which is closed in $X$.
It suffices therefore to show that $\sX'$ is regular at every closed point
of $X$ lying in $\sX_\reg \cap H$. Equivalently, we need to show that
$f \notin \fm^2_{\sX, x}$ for every  closed point
$x \in X$ lying in $\sX_\reg\cap H$. But this is clear.
This concludes the proof of the theorem.
\end{proof}

The following is a variant of  \thmref{thm:B-base-reg} where
the hypersurfaces are required to contain a prescribed closed subscheme of $\sX$
with some necessary conditions.
Let $\sX \subset \P^n_S$ be as in \thmref{thm:B-base-reg} and let
$Z \subset \P^n_S$ be a closed subscheme whose no irreducible component lies in
$\P^n_k$.
Assume furthermore that the following are satisfied. 
\begin{enumerate}
\item
  $\sX_s$ is reduced (this happens, for instance, when $\sX$ is smooth over $S$).
\item
  $Z \cap \ov{\sX}$ is a reduced scheme which is finite and flat over $S$.
\item
  $Z \cap \ov{\sX} \subset \sX_\reg$ and $Z \cap \ov{X} \subset X_\reg$.
\end{enumerate}

\begin{thm}\label{thm:B-base-closed}
Under the above conditions, there exists an integer $d_0 \gg 0$ such that for all
$d \ge d_0$, we can find infinitely many hypersurfaces $H \subset \P^n_S$ of degree
$d$ containing $Z$ for which $\sX_\reg \cap H$ is regular. 
If the generic fiber
of $\sX_\reg$ is smooth, then so is the generic fiber of $\sX_\reg \cap H$.
\end{thm}
\begin{proof}
The proof is completely identical to that
  of \thmref{thm:B-base-reg} modulo the modification that we use
 \propref{prop:B-regularity-*} in place of \lemref{lem:B-regularity}
over the quotient field (and the residue field if it is infinite) of $A$, 
and  use \lemref{lem:Non-empty-sp-Z} in place of \lemref{lem:Non-empty-sp}. 
We need to directly use \thmref{thm:PW-sing-gen} instead of its special case
\corref{cor:PW-sing}.
\end{proof}

\subsection{Bertini for the $(R_a + S_b)$-property over $S$}
\label{sec:B-red*}
%We continue with the notations of \S~\ref{sec:B-reg**}.
The following result extends Theorems~\ref{thm:Bertini-reg-0} and 
~\ref{thm:R-S-main} to schemes over $S$.

\begin{thm}\label{thm:B-Bertini-red}
In the situation of \notbref{notat:BT-13}, assume further that $\sX$ is generically reduced.
Then there exists an integer $d_0 \gg 0$ such that for all $d \ge d_0$,
we can find infinitely many hypersurfaces $H \subset \P^n_S$ of degree
$d$ for which $\sY := \sX \cap H$ satisfies the following.
\begin{enumerate}
\item
The structure map $\sY \to S$ is surjective.
\item
$\sY$ is an effective Cartier divisor on $\sX$. 
\item
$\sY$ does not contain any irreducible component of $\sX_\sing$.
\item
$\sY \cap \sX_s$ is an effective Cartier divisor on $\sX_s$.
\item
  $\sY \cap \sX_\reg$ is regular.
\item
If the generic fiber of $\sX_\reg$ is smooth, 
then so is the generic fiber of  $\sY \cap \sX_\reg$.
\item
If $\sX$ is an $(R_a + S_b)$-scheme for some $a, b \ge 0$, then so is $\sY$.
\item
  If $\sX_\eta$ is irreducible of dimension $m \ge 2$, then $\sY_\eta$ is irreducible.
\end{enumerate}
\end{thm}
\begin{proof}
We shall prove the theorem in several steps.

\vskip .2cm

{\bf{Step~1:}}
Our assumptions imply that the generic and special fibers of $\sX$ are positive dimensional.
Let $\Sigma^b_{\sX} \subset \sX$ be as in ~\eqref{eqn:SR-locus}. 
It follows from \corref{cor:Cat-exm} that $\Sigma^b_{\sX}$ is a finite set if
$\sX$ is an $S_b$-scheme.
Since $\sX$ is generically reduced, $\sX_\sing$ is nowhere dense in $\sX$.
We let $W_0 \subset \sX$ be the set of generic points of the 
following closed subsets of $\sX$:
\begin{listabc}
\item
Irreducible and embedded components of $\sX$.
\item
Irreducible and embedded components of $\sX_s$.
\item
Irreducible components of $\sX_\sing$. 
\end{listabc}
We write $W = W_0 \cup \left(\underset{b | \sX \text{ is } S_b}{\cup} \Sigma^b_{\sX}\right) = (W \cap \sX_\eta) \amalg (W \cap X)$.

\vskip .2cm

{\bf{Step~2:}}
It follows from \thmref{thm:Bertini-reg-0} that there exists
an integer $d_1 \gg 0$ such that for all $d \ge d_1$, we can find a
dense open subscheme $\sU_{1} \subset |H^0(\P^n_K, \sO_{\P^n_K}(d))|$ so that every
$H_\eta \in \sU_{1}(K)$ has the property that it does not meet $W \cap \sX_\eta$,
$\sX_\reg \cap H_\eta$ is regular and $\sX_\eta \cap H_\eta$ is
an $(R_a + S_b)$-scheme if $\sX$ (hence $\sX_\eta$) is so. If $\sX_\reg \cap \sX_\eta$ 
is smooth, then $\sX_\reg \cap H_\eta$ is moreover smooth by 
\cite[Theorem~1]{KA}. 
Since $\dim(\sX_\eta)>0$, \lemref{lem:intersection} below
says that there is a dense open subscheme
$\sU_{2} \subset |H^0(\P^n_K, \sO_{\P^n_K}(d))|$ so that every $H_\eta \in \sU_{2}(K)$
has the property that $H_\eta\cap\sX_{\eta}\neq \emptyset$. Let $\sU=\sU_1\cap\sU_2$.

\vskip .2cm

{\bf{Step~3:}}
If $\sX_\eta$ is irreducible of dimension $m \ge 2$,
it follows from \lemref{lem:B-irr-inf} that there
exists an integer $d'_1 \gg 0$ such that for all $d \ge d'_1$, we can find a
dense open subscheme $\sU' \subset |H^0(\P^n_K, \sO_{\P^n_K}(d))|$ so 
that every $H_\eta \in \sU'(K)$ has the property that
$\sX_\eta \cap H_\eta$ is irreducible.
In the rest of the proof, we shall replace $d_1$ by 
${\rm max}(d_1, d'_1)$ and $\sU$ by $\sU \cap \sU'$ if
$\sX_\eta$ is irreducible of dimension $m \ge 2$.

\vskip .2cm

{\bf{Step~4:}}
If $k$ is infinite, we proved in \thmref{thm:B-base-reg} that there is an
integer $d_2$ such that for all $d\ge d_2$, there is a dense open subscheme
$\sU_3\subset |H^0(\P^n_k, \sO_{\P^n_k}(d))|$ so that every $H_{s}\in \sU_3(k)$ has the
property that if $f \in S'_{s,d}$ is the defining homogeneous polynomial of $H_{s}$, 
then image of $f$ in $\sO_{X,x}$ is not in $\fm^2_{X, x}$ 
for every closed point $x \in X$. Furthermore,  by \lemref{lem:good}
(with $Z=\emptyset$ and $T=(W\cap X)\cup \Delta(X)$)
and \lemref{lem:intersection}, we can find an integer $d_3$ such that for all
$d\ge d_3$, there is a dense open subscheme
$\sU_4\subset |H^0(\P^n_k, \sO_{\P^n_k}(d))|$ so that every $H_{s}\in \sU_4(k)$ has the
property that $H_s\cap X\cap W=\emptyset$, and
$H_s\cap X\neq\emptyset$ if $\dim(X) > 0$.
Take $d\ge d_4={\rm max}(d_2,d_3)$ and $F_d=\sU_3(k)\cap\sU_4(k)$. Then
$F_d \neq \emptyset$.

\vskip .2cm

{\bf{Step~5:}}
If $k$ is finite, we let $Y$ be the finite closed subscheme of $X$ consisting of
closed points in $W\cap X$ with reduced induced subscheme structure and let
$T=\prod_{P\in Y} (\sO_{X,P}/\fm_{X,P}\setminus \{0\}) \subset H^{0}(Y,\sO_{Y})$. Then
Lemma~\ref{lem:Elem} and \corref{cor:PW-sing}, in combination with
\cite[Lemmas~3.1, 3.2]{Poonen-3}, imply that
we can find a subset $\sP\subset S'_s$ of positive density
such that every $f\in \sP$ satisfies the properties that $f|_Y\in T$, $H_f$ does not
meet $(W\cap X)\setminus Y$,
$H_{f}\cap X\neq \emptyset$ if $\dim(X) > 0$
and $f\notin\fm^2_{X,x}$ for every closed point $x\in X\setminus Y$.
But we can then conclude that $H_{f}\cap Y=\emptyset$.
So, $H_{f}\cap W\cap X=\emptyset$ and $f\notin\fm^2_{X,x}$ for every closed point
$x\in X$. It follows that there is an integer $d'_4\gg 0$ such that for every
$d\ge d'_4$, $F_d\coloneq S'_{s,d}\cap\sP\neq\emptyset$.

\vskip .2cm

{\bf{Step~6:}}
We let
\[
  d_0 = \left\{
    \begin{array}{ll}
      \max(d_1, d_4) & \mbox{if $|k| = \infty$} \\
       \max(d_1, d'_4) & \mbox{if $|k| < \infty$.}
    \end{array}
  \right.
\]
Let $d \ge d_0$ and $H_s \in F_d$. It follows from \lemref{lem:Non-empty-sp} that 
${\rm sp}^{-1}(H_s) \cap \sU(K)$ is infinite.
Let $H_\eta \in {\rm sp}^{-1}(H_s) \cap \sU(K)$ be any hypersurface and
let $H \in \P_A(S'_{d})$ be the unique 
hypersurface such that $H_\eta = H \cap \P^n_K$.
We shall show that $\sX \cap H$ satisfies the properties (1) $\sim$ (8)
asserted in the theorem.

Since $\sY \cap \sX_\eta = H_\eta \cap \sX_\eta$ and
$\sY \cap X = H_s \cap X$, it follows from our choice of 
$H_s$ and $H_\eta$ that $\sY \cap \sX_\eta \neq \emptyset$, and $\sY \cap X \neq \emptyset$. This proves (1).  
Since $H \cap W_0 = \emptyset$, the properties (2), (3) and (4) are immediate.
The properties (5) and (6) were proven in \thmref{thm:B-base-reg}, given our choice of
$H$. Using (2), (5) and the fact that $\sY \cap \Sigma^b_\sX = \emptyset$,
the proof of (7) becomes identical to the one given (for the field case) in
the proofs of Theorems~\ref{thm:Bertini-reg-0} and ~\ref{thm:R-S-main}.
The property (8) is clear from the refined choice of $\sU$
if $\sX_\eta$ is irreducible of dimension $m \ge 2$ (see Step~3).
\end{proof}

\begin{lem}\label{lem:intersection}
  Let $K$ be an infinite field and let $X$ be a subscheme of $\P^n_K$ such that
  $\dim (X)>0$. Then for any $d>0$, there is a dense open subscheme
  $\sU_d\subset \P(H^0(\P^n_K,\sO(d)))$ such that $H\cap X\neq \emptyset$
for every $H\in\sU(K)$.
\end{lem}
\begin{proof}
  Since $\dim(X)>0$, it has a subscheme $Y$ such that $Y$ is an integral curve.
  Now if $H\cap Y\neq\emptyset$, then $H\cap X\neq\emptyset$.
  So without loss of generality, we can assume that $X$ is an integral curve.
  Then $W=\ov{X}\setminus X$ is a finite set of closed points. For a point $P\in W$,
  let $\sV_P$ be the closed subscheme of $\P(H^0(\P^n_K,\sO(d)))$
  whose $K$-rational points are the degree $d$ hypersurfaces passing through $P$.
  Let $\sV=\bigcup_{P\in W}V_P$. Then $\sV$ is a proper closed subset of
  $\P(H^0(\P^n_K,\sO(d)))$. Let $\sU=\P(H^0(\P^n_K,\sO(d)))\setminus \sV$. Then for
  any $H\in \sU(K)$, $H$ does not pass through any point of $W$. However, since
  $\ov{X}$ is a closed curve in $\P^n_K$, $H\cap \ov{X}\neq \emptyset$. This implies
  that $H\cap X\neq\emptyset$.
  \end{proof}

\begin{cor}$($Bertini for reducedness$)$\label{cor:B-base-red}
  In the situation of \notbref{notat:BT-13}, assume further that $\sX$ is reduced.
  Then there exists an integer $d_0 \gg 0$ 
such that for all $d \ge d_0$, we can find infinitely many hypersurfaces 
$H \subset \P^n_S$ of degree $d$ for which $\sX \cap H$ is reduced.
\end{cor}
\begin{proof}
Apply \thmref{thm:B-Bertini-red} (7) with $(a,b) = (0,1)$.
\end{proof}

\begin{cor}$($Bertini for normality$)$\label{cor:B-base-normal}
  In the situation of \notbref{notat:BT-13}, assume further that $\sX$ is normal.
  Then there exists an integer $d_0 \gg 0$ 
such that for all $d \ge d_0$, we can find infinitely many hypersurfaces 
$H \subset \P^n_S$ of degree $d$ for which $\sX \cap H$ is normal.
\end{cor}
\begin{proof}
Apply \thmref{thm:B-Bertini-red} (7) with $(a,b) = (1,2)$.
\end{proof}

\begin{cor}$($Bertini for irreducibility$)$\label{cor:B-base-normal-irr}
  In the situation of \notbref{notat:BT-13}, assume further that $\sX$ is irreducible such
  that $\dim(\sX_x) \ge 2$ for $x \in \{\eta, s\}$.
  Then there exists an integer $d_0 \gg 0$ 
such that for all $d \ge d_0$, we can find infinitely many hypersurfaces 
$H \subset \P^n_S$ of degree $d$ for which $\sX \cap H$ is irreducible.
\end{cor}
\begin{proof}
  % If $\sX_\eta = \emptyset$, we apply \thmref{thm:B-irr-fin} and \lemref{lem:Non-empty-sp}. Otherwise,
  We let $H \in \P_A(S'_{d})$ be such that
$\sY := \sX \cap H$ satisfies \thmref{thm:B-Bertini-red} (8).
Then $\sY_\eta$ is irreducible and dense in $\sY$. But this implies that $\sY$ is irreducible.
We note here that the proof of \thmref{thm:B-Bertini-red} (8) does not require
$\sX$ to be generically reduced.
\end{proof}

\begin{cor}$($Bertini for integrality$)$\label{cor:B-base-normal-int}
  In the situation of \notbref{notat:BT-13}, assume further that $\sX$ is integral such that
  $\dim(\sX_x) \ge 2$ for $x \in \{\eta, s\}$.
Then there exists an integer $d_0 \gg 0$ 
such that for all $d \ge d_0$, we can find infinitely many hypersurfaces 
$H \subset \P^n_S$ of degree $d$ for which $\sX \cap H$ is integral.
\end{cor}
\begin{proof}
Combine Corollaries~\ref{cor:B-base-red}
and ~\ref{cor:B-base-normal-irr}.
\end{proof}

\bigskip

\noindent\emph{Acknowledgments.}
The authors would like to thank Qing Liu for reading an earlier version of this
manuscript and sending some very useful comments. They would like to thank the referee
for reading the manuscript very carefully and suggesting many improvements in its
presentation.

\end{document}